\documentclass[12pt]{amsart}
\usepackage{mathrsfs,bbm}

\usepackage[all]{xy}

\newcommand{\C}{\mathbb{C}}
\newcommand{\R}{\mathbb{R}}
\newcommand{\Q}{\mathbb{Q}}
\newcommand{\Z}{\mathbb{Z}}
\newcommand{\N}{\mathbb{N}}

\mathchardef\mhyphen="2D

\renewcommand{\1}{\mathbbm{1}}
\newcommand{\Res}{\text{Res}}

\newcommand{\supp}{\text{supp}}

\newcommand{\arsim}{\xrightarrow{\sim}}
\newcommand{\commsq}[8]{\xymatrix{ #1 \ar[r]^{#5} \ar[d]_{#6} & #2 \ar[d]^{#7} \\ #3 \ar[r]_{#8} & #4 } }
\newcommand{\commtri}[6]{ \xymatrixrowsep{2.5pc}\xymatrixcolsep{1.25pc}
\xymatrix{
#1 \ar[rr]^{#4} \ar[dr]_{#5} & & #2 \ar[dl]^{#6} \\
& #3 & } }
\newcommand{\rcommtri}[6]{ \xymatrixrowsep{1.5pc}\xymatrixcolsep{3pc} \xymatrix{ & #2 \ar[dd]^{#6} \\ #1 \ar[ur]^{#4} \ar[dr]_{#5} & \\ & #3 } }

\newcommand{\ucommtri}[6]{ \xymatrixrowsep{2.5pc}\xymatrixcolsep{1.25pc}
\xymatrix{
& #1 \ar[dl]_{#4} \ar[dr]^{#5} & \\
#2 \ar[rr]_{#6} & & #3 } }

\newcommand{\Spec}{\textnormal{Spec}}

\newcommand{\E}{\mathscr{E}}
\newcommand{\F}{\mathscr{F}}

\newcommand{\p}{\mathfrak{p}}

\newcommand{\NN}{\mathscr{N}}

\newcommand{\OO}{\mathscr{O}}

\newcommand{\Gm}{\mathbb{G}_m}

\newcommand{\Sch}{\mathfrak{Sch}}

\newcommand{\QCoh}{\mathfrak{QCoh}}
\newcommand{\Coh}{\mathfrak{Coh}}

\newcommand{\Vect}{\mathfrak{Vect}}
\newcommand{\ParVect}{\mathfrak{ParVect}}
\newcommand{\Bun}{\textnormal{Bun}}
\newcommand{\HHom}{\textnormal{\underline{Hom}}}
\newcommand{\DD}{\mathbb{D}}
\newcommand{\op}{\textnormal{op}}

\newcommand{\Ob}{\textnormal{Ob}}
\newcommand{\ev}{\textnormal{ev}}

\newcommand{\g}{\mathfrak{g}}
\newcommand{\h}{\mathfrak{h}}

\renewcommand{\t}{\mathfrak{t}}

\newcommand{\Lie}{\textnormal{Lie}}

\newcommand{\Ad}{\textnormal{Ad}}
\newcommand{\ad}{\textnormal{ad}}

\newcommand{\Hom}{\textnormal{Hom}}
\newcommand{\diag}{\textnormal{diag}}

\newcommand{\coker}{\textnormal{coker}}

\newcommand{\pardeg}{\textnormal{par-deg}}

\newcommand{\K}{\mathcal{K}}

\newcommand{\GG}{\mathcal{G}}

\newcommand{\XX}{\mathfrak{X}}
\newcommand{\YY}{\mathfrak{Y}}

\newcommand{\Gmod}{G\mhyphen\mathfrak{mod}}

\newcommand{\f}{\mathfrak{f}}
\newcommand{\PPP}{\mathscr{P}}

\newcommand{\At}{\text{At}}
\newcommand{\mmu}{\boldsymbol{\mu}}
\newcommand{\V}{\mathcal{V}}
\renewcommand{\NN}{\mathcal{N}}
\newcommand{\WW}{\mathcal{W}}

\newtheorem{theorem}{Theorem}[section]
\newtheorem{lem}[theorem]{Lemma}
\newtheorem{cor}[theorem]{Corollary}
\newtheorem{prop}[theorem]{Proposition}

\theoremstyle{definition}
\newtheorem{defn}[theorem]{Definition}
\newtheorem{ex}[theorem]{Example}

\theoremstyle{remark}
\newtheorem{rmk}[theorem]{Remark}

\numberwithin{equation}{section}

\begin{document}

\title{Root stacks, principal bundles and connections}

\author[I. Biswas]{Indranil Biswas}

\address{School of Mathematics, Tata Institute of Fundamental
Research, Homi Bhabha Road, Bombay 400005, India}

\email{indranil@math.tifr.res.in}

\author[S. Majumder]{Souradeep Majumder}

\email{souradip@math.tifr.res.in}

\author[M. L. Wong]{Michael Lennox Wong}

\email{wong@math.tifr.res.in}



\date{}

\begin{abstract}
We investigate principal bundles over a root stack.  In the case of dimension one, we 
generalize the criterion of Weil and Atiyah for a principal bundle to have an algebraic connection.
\end{abstract}

\maketitle

\section*{Introduction}

The notion of a parabolic vector bundle over a compact Riemann surface was first developed in \cite{MS} to obtain a version of the theorem of Narasimhan--Seshadri \cite{NS}, in the case where one wants to describe the moduli space of unitary representations of the fundamental group of the surface with a finite set of punctures.  The extra structure one obtains is the data of a flag in the fibre and a set of real numbers, called weights, at each of the punctures.  When the weights are rational numbers, parabolic vector bundles
can be described as equivariant vector bundles on a suitable Galois cover 
\cite{BiswasOrbifold}, \cite{Boden_thesis}, \cite{NasatyrSteer}. One drawback of this correspondence is that it requires introduction to a new parameter, namely the Galois group for the covering. To remedy this, N.\ Borne has shown that the  category of parabolic vector bundles over a $\C$-scheme with weights lying in  $\tfrac{1}{r} \Z$, with $r \in \N$, is equivalent to the category of vector  bundles over a related object called the ``$r$-th root stack'' which depends only  on the original scheme, the parabolic divisor and the natural number $r$  \cite{Borne}, \cite{Bor2}.  The root stack essentially gives the scheme some ``orbifold structure,'' by putting a cyclic group of order $r$ over the divisor.  This approach of Borne for parabolic bundles has turned out to be very useful (see, for example, \cite{BD}).

A coherent generalization of the notion of a parabolic structure for a principal bundle, even over curves, has been somewhat elusive, largely because it has not been clear what the analogue of a set of weights should be.  The main aim of this paper is to advocate Borne's approach of viewing a parabolic bundle over a quasi-projective variety as a bundle over an associated root stack.  As such, the article begins by defining principal bundles over a smooth algebraic stack over $\C$ and basic constructions, such as associated fibre bundles and reduction of structure group.

One result of the paper, stated in Section \ref{ConnectionCondition}, gives a 
condition for the existence of a connection over a principal bundle over an 
algebraic stack in the style of \cite{AzadBiswas2002}.
To explain this condition,
let $G$ be a reductive affine algebraic group over $\C$. Let $\XX
= \XX_{\OO_X(Z), s, r}$ be a complete root stack of dimension one. We prove that
a principal $G$-bundle $E_G$ over $\XX$ admits a connection if and only if
for any reduction $\F\, \subset E_G$ to a Levi factor $L$ of a parabolic subgroup
of $G$, and any character $\chi : L \to
\C^\times$, the associated line bundle $\F
\times^\chi \C$ satisfies $\deg_\XX \mathcal{M} = 0$. (See
Theorem \ref{connectioncondition}.)

In Section \ref{RootStacks}, we review the construction of a root stack as given in \cite{Cadman}.  We show that in the special case of the Galois covers considered in \cite{BiswasOrbifold}, where all isotropy groups are cyclic of the same order, the associated root stack is in fact the quotient stack, and point out that in the case of a curve, one always has such a realization.

Of course, justifying the root stack approach to parabolic structures necessitates a comparison with existing approaches in the literature, and this is done in Sections \ref{TensorFunctors} and \ref{ParahoricTorsors}.  In his characterization of finite vector bundles \cite{Nori1976}, M.V.\ Nori gave a realization of a principal $G$-bundle over a scheme over an arbitrary field as a tensor functor from the category of finite-dimensional representations of $G$ to the category of vector bundles over the scheme.  One approach that has been taken is that of \cite{BBN}, where a parabolic principal bundle was defined as a tensor functor which takes values in the category of parabolic vector bundles.  Section \ref{TensorFunctors} is concerned with showing that this notion and that of a principal bundle over a root stack are equivalent.

A notion which has appeared in the literature recently (e.g., \cite{PappasRapoport, Heinloth_Uniformization}) is that of a (torsor for a) parahoric Bruhat--Tits group scheme.  The specific instances of this phenomenon which are relevant for us appear in a paper of V.\ Balaji and C.S.\ Seshadri \cite{BalajiSeshadri2012}, where such a torsor is generically a $G$-bundle.  They show that equivariant $G$-bundles for a Galois cover correspond to parahoric torsors on the base.  In their description, it is the isotropy representation (in $G$) over the ramification points of the Galois cover which determines the appropriate Bruhat--Tits group scheme (i.e., the analogue of the flag type for parabolic vector bundles).  Such representations may be thought of as restrictions of cocharacters of the cover, and hence as rational cocharacters on the base.  It is this that yields the analogous notion of a set of weights for parabolic bundle/parahoric torsor (see Section \ref{PGST}).  This was already suggested by P.\ Boalch in his local classification of connections on $G$-bundles for reductive groups \cite{Boalch_Parahoric}.  The aim of Section \ref{ParahoricTorsors} is to show that these ideas are all readily expressible in terms of principal bundles over root stacks.  Specifically, we define the local type of a principal bundle over the root stack and show that these correspond to parahoric torsors of a given type.  Finally, we restrict Boalch's definition of a logarithmic parahoric connection using a condition paralleling the one for parabolic vector bundles (e.g., as in \cite[\S2.2]{BiswasLogares2011}) and show that one has a correspondence between connections on a principal bundle over the root stack and connections on the parahoric torsor.

MLW would like to thank V.\ Balaji for some helpful clarifications and for sharing a draft version of \cite{BalajiSeshadri2012}.  He is also grateful for the support of the Fonds qu\'eb\'ecois de la recherche sur la nature et les technologies in the form of a Bourse de recherche postdoctorale (B3).

\section{Principal Bundles on Algebraic Stacks} \label{pbstack}

We will work over the category of $\C$-schemes, which we denote by $\Sch/\C$.  If not otherwise indicated, $\XX$ will be a smooth algebraic stack locally of finite type over $\C$.  We will also fix a complex algebraic group $G$.

For a $\C$-scheme $U$, the fibre category of $\XX$ over $U$ will be denoted by $\XX(U)$.  Via the $2$-Yoneda lemma, we will freely identify an object $\f \in \Ob \, \XX(U)$ with a $1$-morphism $\f : U \to \XX$.

\subsection{Coherent Sheaves} \label{coherentsheaves}

A \emph{coherent sheaf} $\V$ on $\XX$ consists of the following data (e.g.\ \cite[Definition 7.18]{Vistoli_IntersectionTheory}, \cite[Definition 2.50]{Gomez_Stacks}, \cite[Lemme 12.2.1]{LMB}).  If $\f : U \to \XX$ is a smooth atlas (i.e., if $U$ is a $\C$-scheme and $\f$ is a smooth map), then we have a coherent $\OO_U$-module $\V_\f$.  For a (2-)commutative diagram
\begin{align} \label{overXX}
\vcenter{
\commtri{ U }{ V }{ \XX }{ k }{ \f }{ \g }
}
\end{align}
with $\f, \g$ smooth atlases, we are given an isomorphism
\begin{align} \label{sheafiso}
\alpha_k^\V = \alpha_k : \V_\f \arsim k^* \V_\g,
\end{align}
such that for a (2-)commutative diagram
\begin{align} \label{compositionoverXX}
\vcenter{
\xymatrix{
U \ar[dr]_{\f} \ar[r]^k & V \ar[d]^\g \ar[r]^l & W \ar[dl]^ \h \\
& \XX & } }
\end{align}
the diagram
\begin{align} \label{compatibilityconditionCS}
\vcenter{
\xymatrix{
\V_\f \ar[r]^-{\alpha_{l \circ k}} \ar[d]_{\alpha_k} & (l \circ k)^* \V_\h \ar@{=}[d] \\
k^* \V_\g \ar[r]_{k^* \alpha_l} & k^* l^* \V_\h }
}
\end{align}
commutes, where the two objects on the right side are identified via the canonical isomorphism of functors $(l \circ k)^* \arsim k^* l^*$.

We will call a coherent sheaf $\V$ on $\XX$ a \emph{vector bundle} if $\V_\f$ is a locally free $\OO_U$-module whenever $\f \in \Ob \, \XX(U)$.  

If $\XX$ is a Deligne--Mumford stack, then it is enough to specify $\V_\f$ for \'etale atlases $\f : U \to \XX$.  In this case, we may define the sheaf of differentials $\Omega_\XX^1 = \Omega_{\XX/\C}^1$ as follows.  For an \'etale morphism $\f : U \to \XX$, we simply set
\begin{align*}
\Omega_{\XX, \f}^1 := \Omega_{U/\C}^1.
\end{align*}
Given a diagram (\ref{overXX}), with $\f$ and $\g$, and hence $k$, \'etale, one 
has an exact sequence \cite[Morphisms of Schemes, Lemma 32.16]{StacksProject}
\begin{align*}
0 \to k^* \Omega_{V/\C}^1 \to \Omega_{U/\C}^1 \to \Omega_{U/V}^1 \to 0;
\end{align*}
since $k$ is \'etale, the last term vanishes, so we obtain isomorphisms (\ref{sheafiso}).  The fact that they satisfy the compatibility condition (\ref{compatibilityconditionCS}) is due to their canonical nature.

If $\V$ is a vector bundle over a Deligne--Mumford stack $\XX$, by a \emph{connection on $\V$} we will mean the data of a connection $\nabla_\f$ on $\V_\f$ for each $\f \in \Ob \, \XX(U)$ such that for a diagram (\ref{overXX}), the following commutes:
\begin{align*}
\xymatrixcolsep{5pc}
\xymatrix{\V_\f \ar[r]^{\nabla_\f} \ar[d]_{ \alpha_k^\V } & \V_\f \otimes_{\OO_U} \ar[d]^{ \alpha_k^\V \otimes \alpha_k^{\Omega} } \\
k^* \V_\g \ar[r]_-{ k^* \nabla_\g } & k^* \V_\g \otimes_{\OO_U} k^* \Omega_{V/\C}^1 }
\end{align*}

\subsection{Principal Bundles} \label{principalbundles}

With $\XX$ and $G$ as in the beginning of the section, we can take as a definition of a \emph{principal $G$-bundle on $\XX$} the following, paralleling the definition of a coherent sheaf.  For each smooth atlas $\f : U \to \XX$, we are given the data of a principal $G$-bundle $\E_\f$ over $U$, and for each diagram (\ref{overXX}) we have isomorphisms
\begin{align*}
\beta_k = \beta_k^\E : \E_\f \arsim k^* \E_\g
\end{align*}
such that for a diagram (\ref{compositionoverXX}), one has a commutative diagram
\begin{align} \label{compatibilitycondition}
\vcenter{
\xymatrix{
\E_\f \ar[r]^-{\beta_{l \circ k}} \ar[d]_{\beta_k} & (l \circ k)^* \E_\h \ar@{=}[d] \\
k^* \E_\g \ar[r]_{k^* \beta_l} & k^* l^* \E_\h, } }
\end{align}
where again, we use the canonical isomorphism $(l \circ k)^* \arsim k^* l^*$ to identify the two bundles on the right hand side.

Let $\F$ be another $G$-bundle over $\XX$.  A \emph{morphism} of $G$-bundles $\varphi : \E \to \F$ consists of the data of a morphism $\varphi_\f : \E_\f \to \F_\f$ for each smooth atlas $\f : U \to \XX$, such that given a diagram (\ref{overXX}), the square
\begin{align*}
\commsq{ \E_\f }{ \F_\f }{ k^* \E_\g }{ k^* \F_\g }{ \varphi_\f }{ \beta_k^\E }{ \beta_k^\F }{ k^*\varphi_\g }
\end{align*}
commutes.  It is not hard to see then that the category of principal $G$-bundles over $\XX$ is a groupoid.

We will recall that the classifying stack $BG$ is the fibred category whose objects over a $\C$-scheme $U$ are principal $G$-bundles over $U$ and whose morphisms are pullback diagrams of $G$-bundles.  The following is not hard to verify.

\begin{lem} \label{GbundleBG}
The datum of a principal $G$-bundle over $\XX$ in the above sense is equivalent to the datum of a morphism $\XX \to BG$.  Two $G$-bundles over $\XX$ are isomorphic if and only if the corresponding morphisms $\XX \to BG$ are $2$-isomorphic.
\end{lem}

Given $\XX, G$ as above, we may now consider the fibred category whose objects over a $\C$-scheme $U$ are $G$-bundles $E \to \XX \times U$ and whose morphisms are pullback diagrams of $G$-bundles.  We will denote this category by 
\begin{align*}
\Bun_G \XX.
\end{align*}

Let $\XX, \YY$ be separated algebraic stacks of finite presentation over $\C$ 
with finite diagonals.  We may consider the fibred category $\HHom_\C( \XX, 
\YY)$ whose fibre category over a $\C$-scheme $U$ is the groupoid of functors $\Hom_U(\XX \times U, \YY \times U) = \Hom_\C (\XX \times U, \YY)$.

Now, taking $\XX \times U$ in Lemma \ref{GbundleBG} and $\YY := BG$, it is not hard to see that the following holds.

\begin{prop}
There is an equivalence
\begin{align*}
\Bun_G \XX \cong {\HHom}_\C(\XX, BG).
\end{align*}
\end{prop}

\subsection{Principal Bundles on Quotient Stacks}

Let $Y$ be $\C$-scheme and let $\Gamma$ be a complex algebraic group acting on 
$Y$ on the left with action map $\lambda : \Gamma \times Y \to Y$.  Let
$p_\Gamma$ and $p_Y$ be the projections of $\Gamma \times Y$ to $p_\Gamma$ and
$p_Y$ respectively.  Later, we 
will primarily be concerned with the case where $\Gamma$ is finite, but what we 
record here holds in greater generality.

Let $\pi : E \to Y$ be a (right) $G$-bundle over $Y$.  Suppose $\Lambda : \Gamma \times E \to E$ is a left action for which $\pi$ is $\Gamma$-equivariant and which commutes with the $G$-action $\rho : E \times G \to G$, i.e.,
\begin{align*}
 \xymatrix{
\Gamma \times E \times G \ar[r]^-{ \1_\Gamma \times \rho } \ar[d]_{ \Lambda \times \1_G } & \Gamma \times E \ar[d]^{ \Lambda } \\
E \times G \ar[r]_-{\rho} & E }
\end{align*}
commutes. In this case, we call $\Lambda$ a \emph{compatible $\Gamma$-action}.  A 
\emph{$(\Gamma, G)$-bundle} is a $G$-bundle on $Y$ together with a compatible 
$\Gamma$-action.  If $E \to Y$ and $F \to Y$ are $(\Gamma, G)$-bundles, a 
\emph{morphism of $(\Gamma, G)$-bundles} $E \to F$ is a morphism of $G$-bundles which commutes with the $\Gamma$-action.  We will denote the stack of $(\Gamma, G)$-bundles by
\begin{align*}
\Bun_{\Gamma, G} Y.
\end{align*}
This is the fibred category whose fibre category over a $\C$-scheme $U$ is the groupoid of $(\Gamma, G)$-bundles over $Y \times U$, where $Y \times U$ has the $\Gamma$-action induced from $\lambda$.

\begin{rmk}
To give a compatible $\Gamma$-action on a $G$-bundle $E$ is equivalent to giving an isomorphism $\tau : p_Y^* E \arsim \lambda^* E$ of $G$-bundles over $\Gamma \times Y$ such that the following ``cocycle condition'' holds.  Consider the diagrams
\begin{align*} 
& \xymatrix{
\Gamma \times \Gamma \times Y \ar[r]^-{\1_\Gamma \times \lambda} \ar[dr]^L \ar[d]_{m \times \1_X} & \Gamma \times Y \ar[d]^\lambda \\
\Gamma \times Y \ar[r]_\lambda & Y, } & 
& \xymatrix{
\Gamma \times \Gamma \times Y \ar[r]^-{p_{\Gamma \times Y} } \ar[dr]^N \ar[d]_{\1_\Gamma \times \lambda} & \Gamma \times Y \ar[d]^\lambda \\
\Gamma \times Y \ar[r]_{p_Y} & Y, } & 
& \xymatrix{
\Gamma \times \Gamma \times Y \ar[r]^-{p_{\Gamma \times Y} } \ar[dr]^\Pr \ar[d]_{m \times \1_X} & \Gamma \times Y \ar[d]^{p_Y} \\
\Gamma \times Y \ar[r]_{p_Y} & Y. } 
\end{align*}
Then, modulo canonical isomorphisms, we get two isomorphisms $$(m \times \1_Y)^* 
\tau, (\1_\Gamma \times \lambda)^* \tau \circ p_{\Gamma \times Y}^* \tau\,:\, 
\Pr^* E \,\arsim \,L^* E\, .$$  The condition we require is that
\begin{align} \label{actioncocycle}
(m \times \1_Y)^* \tau = (\1_\Gamma \times \lambda)^* \tau \circ p_{\Gamma \times Y}^* \tau.
\end{align}
In what follows, we will more often talk about $(\Gamma, G)$-bundles in terms of this description.
\end{rmk}

Recall that the quotient stack $[ \Gamma \backslash Y ]$ is the fibred category whose objects over a $\C$-scheme $U$ are diagrams
\begin{align*}
\xymatrix{ M \ar[r] \ar[d] & Y \\ U, & }
\end{align*}
where the vertical arrow is a (left) principal $\Gamma$-bundle over $U$ and the horizontal arrow is a $\Gamma$-equivariant map, and whose morphisms over a morphism $U \to V$ of $\C$-schemes are diagrams
\begin{align*}
\xymatrix{ M \ar[r] \ar[d] & N \ar[r] \ar[d] & Y \\ U \ar[r] & V, & }
\end{align*}
where the square is Cartesian and the composition $M \to N \to Y$ is the same 
arrow as occurring the object over $U$.  There is a natural quotient morphism $\nu : Y \to [ \Gamma \backslash Y ]$ which takes a $\C$-morphism $f : U \to Y$ to 
\begin{align*}
\xymatrix{ \Gamma \times U \ar[r] \ar[d] & Y \\ U, & }
\end{align*}
where the horizontal map is $\lambda \circ (\1_\Gamma \times f)$.

With this, the diagram
\begin{align*}
\commsq{ \Gamma \times Y }{ Y }{ Y }{ [ \Gamma \backslash Y ] }{ \lambda }{ p_Y }{ \nu }{ \nu }
\end{align*}
is Cartesian, yielding an isomorphism 
\begin{align} \label{YGamma}
Y \times_{[ \Gamma \backslash Y ]} Y \cong \Gamma \times Y.
\end{align}
Hence we obtain isomorphisms
\begin{align} \label{triple}
Y \times_{[ \Gamma \backslash Y ]} Y \times_{[ \Gamma \backslash Y ]} Y \cong \Gamma \times Y \times_{[ \Gamma \backslash Y ]} Y \cong \Gamma \times \Gamma \times Y.
\end{align}

Let $\E$ be a $G$-bundle over $[ \Gamma \backslash Y ]$.  As $\nu : Y \to [ \Gamma \backslash Y ]$ is a smooth atlas for $[ \Gamma \backslash Y ]$, to it is associated a $G$-bundle $\E_\nu$.  The data of the bundle $\E$ and the diagram 
\begin{align*}
 \xymatrix{
Y \times_{[ \Gamma \backslash Y ]} Y \ar[r]^-{p_2} \ar[d]_{p_1} & Y \ar[d]^{\nu} \\
Y \ar[r]_{\nu} & [ \Gamma \backslash Y ] }
\end{align*}
yield isomorphisms $\beta_{p_i} : \E_{\nu \circ p_i} \arsim p_i^* \E_\nu$.  Since $\E_{\nu \circ p_1} \cong \E_{\nu \circ p_2}$, we then obtain an isomorphism $\sigma : p_1^* \E_\nu \arsim p_2^* \E_\nu$ of $G$-bundles over $Y \times_{[ \Gamma \backslash Y ]} Y$.  Now, if $p_{ij} : Y \times_{[ \Gamma \backslash Y ]} Y \times_{[ \Gamma \backslash Y ]} Y \to Y \times_{[ \Gamma \backslash Y ]} Y$ are the various projections, then the condition (\ref{compatibilitycondition}) implies that the cocycle condition
\begin{align} \label{quotientcocycle}
p_{13}^* \sigma = p_{23}^* \sigma \circ p_{12}^* \sigma.
\end{align}
holds

Under the isomorphism (\ref{YGamma}), $\sigma$ becomes an isomorphism $\tau : p_Y^* \E_\nu \arsim \lambda^* \E_\nu$ of $G$-bundles over $\Gamma \times Y$, and the condition (\ref{quotientcocycle}) above translates precisely into the condition (\ref{actioncocycle}), and hence $\E_\nu$ is a $(\Gamma, G)$-bundle over $Y$.

Conversely, let $E$ be a $(\Gamma, G)$-bundle on $Y$.  Let $\f : U \to [ \Gamma \backslash Y ]$ be any smooth atlas and consider the diagram
\begin{align*}
 \xymatrix{
U \times_{ [ \Gamma \backslash Y ]} Y \ar[r]^-{ q_\f } \ar[d]_{ \nu_\f } & Y \ar[d]^{\nu} \\
U \ar[r]_{ \f } & [ \Gamma \backslash Y ] }
\end{align*}

Then $\nu_\f : U \times_{ [ \Gamma \backslash Y ] } Y \to U$ is a $\Gamma$-torsor and $q_\f^* E$ is a $(\Gamma, G)$-bundle over $U \times_{[ \Gamma \backslash Y ]} Y$.  Thus $q_\f^*E$ descends to a $G$-bundle $\E_\f$ over $U$.  For a diagram (\ref{overXX}), the morphisms $\beta_k$ arise by considering the diagram
\begin{align*}
\xymatrix{
U \times_{[ \Gamma \backslash Y ]} Y \ar[r]^{\widetilde{k}} \ar[d]_{\nu_\f} & V 
\times_{[ \Gamma \backslash Y ]} Y \ar[r]^-{q_\g} \ar[d]_{\nu_\g} & Y \ar[d]^\nu \\
U \ar[r]_k & V \ar[r]_\g & [ \Gamma \backslash Y ], }
\end{align*}
and the uniqueness of the descended objects.  Therefore a $(\Gamma, G)$-bundle on $Y$ yields a $G$-bundle on $[ \Gamma \backslash Y ]$.

\begin{prop} \label{stackGamma}
There is an equivalence
\begin{align*}
\Bun_{\Gamma, G} Y \arsim \Bun_G [ \Gamma \backslash Y ].
\end{align*}
\end{prop}

\begin{proof}
To see this, for each $\C$-scheme $U$, we need to repeat the above argument with $Y$ replaced by $Y \times U$ with the induced action.  One will need to use the fact that
\begin{align*}
[\Gamma \backslash Y \times U] \cong [\Gamma \backslash Y ] \times U,
\end{align*}
which is not hard to prove.
\end{proof}

\subsection{Associated Bundles} \label{associatedbundles}

For a principal $G$-bundle $E$ over a $\C$-scheme $X$, and a left $G$-action $\lambda : G \times F \to F$ on a $\C$-scheme $F$, we will denote the associated fibre bundle over $X$ with fibre $F$ by
\begin{align*}
E \times^\lambda F.
\end{align*}

\begin{lem} \label{associatedbundleiso}
Let $f : X \to Y$ be a morphism of $\C$-schemes, let $E$ be a principal 
$G$-bundle over $Y$, and let $F$ be a $\C$-scheme endowed with a (left) 
$G$-action $\lambda : G \times F \to F$.  Then there is a canonical isomorphism
\begin{align*}
\nu_f^{E, \lambda} : f^*E \times^\lambda F \arsim f^*( E \times^\lambda F)
\end{align*}
of schemes over $X$.
\end{lem}

\begin{proof}
By definition, the diagram
\begin{align*}
\commsq{ f^* P }{ P }{ X }{ Y }{}{}{}{}
\end{align*}
is Cartesian.  Hence, so are
\begin{align*}
& \vcenter{ \commsq{ f^* P \times F }{ P \times F }{ X }{ Y }{}{}{}{} } & \text{ and } & & \vcenter{ \commsq{ f^* P \times^\lambda F }{ P \times^\lambda F }{ X }{ Y, }{}{}{}{} }
\end{align*}
the latter obtained by taking the quotient of the top row of the diagram on the left by $G$.  Also, by definition,
\begin{align*}
\commsq{ f^* (P \times^\lambda F) }{ P \times^\lambda F }{ X }{ Y, }{}{}{}{}
\end{align*}
is Cartesian, so there is a canonical isomorphism between the corresponding 
upper-left corners.
\end{proof}

We now define for a smooth morphism $\f : U \to \XX$, where $U$ is a scheme,
\begin{align*}
(\E \times^\lambda F)_\f := \E_\f \times^\lambda F.
\end{align*}
Given a diagram (\ref{overXX}), the data for $\E$ comes with an isomorphism $\beta_k : \E_\f \to k^*\E_\g$, and hence there is an induced isomorphism 
\begin{align*}
\E_\f \times^\lambda F \xrightarrow{\overline{\beta}_k} k^*\E_\g \times^\lambda 
F.
\end{align*}
Composing this with the isomorphism $\nu_k^{\E_g, \lambda}$ provided by the Lemma, we obtain isomorphisms
\begin{align*}
\nu_k^{\E_g, \lambda} \circ \overline{\beta}_k : (\E \times^\lambda F)_\f \arsim 
k^*(\E \times^\lambda F)_\g.
\end{align*}
These will satisfy the relations (\ref{compatibilitycondition}) because the $\beta_k$ do and because of the canonical nature of the isomorphisms obtained in Lemma \ref{associatedbundleiso}.

As usual, we are mainly interested in this construction in the cases where $F = 
V$ is a representation of $G$, via $\rho : G \to {\rm GL}(V)$, say, in which case 
the 
associated bundle $\E \times^\rho V$ is a vector bundle, and where $F = H$ is another algebraic group on which $G$ acts via a homomorphism $\varphi : G \to H$ (and left multiplication), yielding a principal $H$-bundle $\E \times^\varphi H$.

\begin{cor}
If $V$ is a finite-dimensional vector space and $\rho : G \to {\rm GL}(V)$ a 
representation, then $\E \times^\rho V$ is a vector bundle in the sense of Section \ref{coherentsheaves}.  If $H$ is a complex algebraic group and $\varphi : G \to H$ a homomorphism, then $\E \times^\varphi H$ is a principal $H$-bundle in the sense of Section \ref{principalbundles}.
\end{cor}

In particular, the adjoint bundle
\begin{align*}
\ad \, \E := \E \times^\ad \Lie(G).
\end{align*}
arising from the adjoint representation $\ad : G \to {\rm GL}(\Lie(G))$ of $G$ on 
its Lie algebra $\Lie(G)$ is well-defined.

\begin{rmk} \label{associatedbundleBG}
A homomorphism of algebraic groups $\varphi : G \to H$ induces a $1$-morphism of algebraic stacks $B\varphi : BG \to BH$, taking a principal $G$-bundle over $U$ (an object of $BG$) to the associated $H$-bundle.  One then sees that if the principal $G$-bundle $\E$ on $\XX$ corresponds to the morphism $E : \XX \to BG$ (via the equivalence of Lemma \ref{GbundleBG}), then the $H$-bundle $\E \times^\varphi H$ corresponds to $B \varphi \circ E$.
\end{rmk}

\subsection{Reduction of Structure Group}

Fix a principal $G$-bundle $\E$ over $\XX$ and let $H \subseteq G$ be a closed subgroup of $G$.  Then for each $\f \in \Ob \, \XX(U)$ and each diagram (\ref{overXX}), the isomorphism $\beta_k$ induces another one 
\begin{align*}
\E_\f /H \arsim k^* \E_\g/H = k^*( \E_\g/H),
\end{align*}
which we will denote by $\beta_{k,H}$.  Then a \emph{reduction $\tau$ of the structure group to $H$} consists of the data of a section $\tau_\f : U \to \E_\f/H$ for each $\f \in \Ob \, \XX(U)$ such that for each diagram (\ref{overXX}), one has
\begin{align} \label{reductioncompatibility}
\overline{\beta}_{k,H} \circ \tau_\f = k^* \tau_\g.
\end{align}

\begin{lem} \label{reductionequivalence}
Let $\E$ be a principal $G$-bundle over $\XX$.  The following are equivalent pieces of information:
\begin{enumerate}
\item[(a)] a reduction of structure group $\tau$ to $H$;
\item[(b)] a principal $H$-bundle $\F$ and an isomorphism
    \begin{align*}
    \phi : \E \cong \F \times^\iota G,
    \end{align*}
    where $\iota : H \to G$ is the inclusion map;
\item[(c)] a factorization 
    \begin{align*}
    \commtri{ \XX }{ BH }{ BG, }{ F }{ E }{ B\iota}
    \end{align*}
    where $E$ is the morphism of stacks corresponding to $\E$.
\end{enumerate}
\end{lem}

\begin{proof}
Suppose we are given a reduction $\tau$ of $\E$ to $H$.  For $\f \in \Ob \, 
\XX(U)$, we consider the fibre product $\F_\f := \E_\f \times_{\E_\f/H} U$ which 
fits into a Cartesian diagram
\begin{align} \label{reduction}
\vcenter{
\xymatrix{
 \F_\f \ar[r] \ar[d] & \E_\f \ar[d] \\
U \ar[r]_-{\tau_\f} & \E_\f/H. }
}
\end{align}
Then $\F_f$ is an $H$-bundle over $U$ with the property that $\F_f \times^\iota G \cong \E_\f$.  To see that we get isomorphisms $\beta_k^\F : \F_\f \to k^* \F_\g$, we take the diagram (\ref{reduction}) for $\g \in \Ob \, \XX(V)$ and pull it back via $k$.  Then using the $\beta_k^\E$ and (\ref{reductioncompatibility}), we observe that $\F_\f$ and $k^* \F_\g$ give isomorphic fibre products, so we may construct the $\beta_k^\F$.  The compatibility condition (\ref{compatibilitycondition}) will come from that of the $\beta_k^\E$.

Conversely, suppose there exists an $H$-bundle $\F$ as in (b).  Now, noting that $\F_\f \to \E_\f$ is an $H$-equivariant morphism and $\F_\f/H \cong U$, we get sections $\tau_\f : U \to \E_\f/H$.  To see that we have (\ref{reductioncompatibility}), we use the diagram
\begin{align*}
\commsq{ \F_\f }{ \E_\f }{ k^* \F_\g }{ k^* \E_\g. }{}{ \beta_k^\F }{ \beta_k^\E }{}
\end{align*}
This shows that (a) and (b) are equivalent.

The equivalence of (b) and (c) is clear from Remark \ref{associatedbundleBG}.
\end{proof}

\subsection{Connections}

Let $\E$ be a principal bundle over a Deligne--Mumford stack $\XX$.  Then a 
\emph{connection} $\nabla$ on $\E$ consists of the data of a connection 
$\nabla_\f$ on each $\E_\f$ where $\f : U \to \XX$ is an \'etale atlas which 
pulls back properly with respect to diagrams (\ref{overXX}). 
To be precise, suppose we realize the connections in 
terms of $\Lie(G)$-valued 
$1$-forms so that $\omega_\f \in H^0(U, \Omega_{U/\C}^1 \otimes \Lie(G))$ is the connection $1$-form corresponding to $\nabla_\f$.  From the Cartesian diagram
\begin{align*}
\commsq{ k^* \E_\g }{ \E_\g }{ U }{ V }{ \overline{k} }{}{}{ k },
\end{align*}
we obtain a connection $\overline{k}^* \omega_\g$ on $k^* \E_\g$.  Then the 
condition that we want is
\begin{align} \label{connectioncompatibilitycondition}
\omega_\f = \beta_k^* \overline{k}^* \omega_\g.
\end{align}

One obtains the following simply because induced connections behave well with respect to pullbacks.

\begin{lem} \label{inducedconnection}
Let $\E$ be a principal $G$-bundle admitting a connection $\nabla$. 
\begin{enumerate}
\item[(a)] If $\rho : G \to {\rm GL}(V)$ is a representation, then there is an 
induced 
connection $\nabla^\rho$ on the associated vector bundle $\E \times^\rho V$.
\item[(b)] If $\varphi : G \to H$ is a homomorphism of algebraic groups, then there is an induced connection $\nabla^\varphi$ on the associated principal bundle $\E \times^\varphi H$.
\end{enumerate} 
\end{lem}

The following statement is justified in the course of the proof of 
\cite[Proposition 2.3]{AzadBiswas2002}.

\begin{lem} \label{Levisplitting}
Let $G$ be a reductive algebraic group.  Let $L \subseteq G$ be the Levi factor 
of a parabolic subgroup of $G$.  Then there is an $L$-equivariant splitting 
$\psi : \Lie(G) \to \Lie(L)$.
\end{lem}

\begin{lem} \label{splittingconnection}
Let $L \subseteq G$ be any closed subgroup.  Assume that there exists an $H$-equivariant (for the adjoint action) splitting $\psi : \Lie(G) \to \Lie(H)$ (of the inclusion map $\Lie(H) \hookrightarrow \Lie(G)$).  If the $G$-bundle $\E$ admits a connection and a reduction to $H$, then the resulting $H$-bundle (as given in Lemma \ref{reductionequivalence}) admits a connection. 
\end{lem}

\begin{proof}
This is the analogue of Lemma 2.2 of \cite{AzadBiswas2002}.  Let $\F$ be the $H$-bundle arising from the reduction in structure group and for $\f \in \Ob \, \XX$ let $j_\f : \F_\f \to \E_\f$ be the inclusion morphism.  Given a connection form $\omega_\f$ on $\E_\f$, the corresponding connection form on $\F_\f$ is given in the quoted result by
\begin{align*}
\nu_\f := \psi \circ j_\f^* \omega_\f.
\end{align*}
The compatibility condition (\ref{connectioncompatibilitycondition}) can be obtained by tracing through the diagram
\begin{align*}
\xymatrix{
\F_\f \ar[r]^{\beta^\F} \ar[d]_{j_\f} & k^* \F_\g \ar[r]^-{\overline{k}_\F} 
\ar[d]^{k^* j_\g} & \F_\g \ar[d]^{j_\g} \\
\E_\f \ar[r]^{\beta_k^\E} \ar[dr] & k^* \E_\g \ar[d] \ar[r]_-{\overline{k}_\E} & 
\E_g \ar[d] \\
& U \ar[r] & V. }
\end{align*}
\end{proof}

\subsection{The Atiyah Sequence}

Let $Y$ be a smooth $\C$-scheme (locally) of finite type and let $\pi : E \to Y$ be a principal $G$-bundle. One has an exact $G$-equivariant sequence of vector bundles
\begin{align*}
0 \to E \times \g \to TE \to \pi^* TY \to 0
\end{align*}
and the Atiyah sequence can be obtained by quotienting by the $G$-action:
\begin{align} \label{Atiyahsequence}
0 \to \ad \, E \to \At \, E \to TY \to 0.
\end{align}

\begin{lem} \label{Atiyah}
Let $f : X \to Y$ be an \'etale morphism of smooth $\C$-schemes (locally) of finite type and let $\pi : E \to Y$ be a principal $G$-bundle.  Then there is a canonical isomorphism
\begin{align*}
\gamma_f^E : \At \, f^* E \arsim f^* \At \, E
\end{align*}
fitting into a commutative diagram
\begin{align} \label{Atiyahsequencepullback}
\vcenter{
\xymatrix{
0 \ar[r] & \ad \, f^*E \ar[r] \ar[d]^{\nu_f^\ad} & \At \, f^* E \ar[r] \ar[d]^{\gamma_f^E} & TU \ar[r] \ar[d] & 0 \\
0 \ar[r] & f^* \ad \, E \ar[r] & f^* \At \, E \ar[r] & f^* TV \ar[r] & 0, } }
\end{align}
where $\nu_f^\ad$ is the isomorphism given in Lemma \ref{associatedbundleiso}, 
noting that $\ad \, E$ is the bundle associated to the adjoint representation 
$\ad : G \to {\rm GL}(\g)$.
\end{lem}

Observe that the bottom row, being the pullback of an exact sequence by a flat map, is exact.

\begin{proof}
We begin with the Cartesian diagram
\begin{align*}
\commsq{ f^* E }{ E }{ X }{ Y }{ \overline{f} }{ \overline{\pi} }{ \pi }{ f }
\end{align*}
and observe that since $\overline{f}$ is obtained from $f$ by base change it is 
\'etale.  Now, we have a canonical exact sequence
\begin{align*}
0 \to \overline{f}^* \Omega_{E/\C}^1 \to \Omega_{f^*E/\C}^1 \to \Omega_{f^*E/E}^1 
\to 0
\end{align*}
whose last term vanishes since $\overline{f}$ is \'etale.  Thus, we have a 
canonical 
isomorphism $T(f^*E) \arsim f^* TE$, and hence a Cartesian square
\begin{align*}
\commsq{ T f^*E }{ TE }{ X }{ Y.}{}{}{}{}
\end{align*}
Quotienting by the action of $G$ on the top row yields another Cartesian square
\begin{align*}
\commsq{ \At \, f^*E }{ \At \, E }{ X }{ Y,}{}{}{}{}
\end{align*}
whence the canonical isomorphism $\gamma_f^E : \At \, f^*E \arsim f^* \At \, E$.

To see that the first square in (\ref{Atiyahsequencepullback}) commutes, we first note that the proof of Lemma \ref{associatedbundleiso} actually shows that $\ad \, f^*E$ and $f^* \ad \, E$ are both canonically isomorphic to $\ad \, E \times_Y X$ and the argument above similarly shows that $\At \, f^*E$ and $f^* \At \, E$ are both canonically isomorphic to $\At \, E \times_Y X$.  Now, modulo canonical isomorphisms, the horizontal maps are essentially the base change map $\ad \, E \times_Y X \to \At \, E \times_Y X$, and so the first square must commute.  The second square commutes since it is map of cokernels and because of the canonical nature of the morphism $TU \to f^* TV$.
\end{proof}

We will now assume that $\XX$ is a Deligne--Mumford stack and let $\E$ be a principal $G$-bundle over $\XX$.  Given an \'etale map $\f : U \to \XX$, we may define
\begin{align*}
(\At \, \E)_\f := \At \, \E_\f.
\end{align*}
For a diagram (\ref{overXX}) in which $\f, \g$ and hence $k$ are \'etale, we can compose the isomorphism $\At \, \E_\f \to \At \, k^* \E_\g$ with the isomorphism $\At \, k^* \E_\g \to k^* \At \, \E_\g$ obtained in Lemma \ref{Atiyah} to get isomorphisms
\begin{align*}
(\At \, \E)_\f \arsim k^* (\At \, \E)_\g.
\end{align*}
Arguing as in Section \ref{associatedbundles}, these will satisfy the condition in (\ref{compatibilityconditionCS}) and hence yield a well-defined vector bundle on $\XX$.  In fact, since the diagram
\begin{align*}
\xymatrix{
0 \ar[r] & \ad \, \E_\f \ar[r] \ar[d] & \At \, \E_\f \ar[r] \ar[d] & TU \ar[d] \ar[r] & 0 \\
0 \ar[r] & k^* \ad \, \E_\g \ar[r] & k^* \At \, \E_\g \ar[r] & k^*TV \ar[r] & 0 }
\end{align*}
commutes, the sequence
\begin{align*}
0 \to \ad \, \E \to \At \, \E \to T \XX \to 0
\end{align*}
is well-defined as an exact sequence of vector bundles on $\XX$.  We will call it the \emph{Atiyah sequence associated to $\E$}.

\begin{lem} \label{Atiyahcondition}
A principal bundle $\E$ over a Deligne--Mumford stack $\XX$ admits a connection if and only if its Atiyah sequence splits.
\end{lem}

The only question here is the compatibility with the isomorphisms $\beta_k$, which is built into the definition of a connection.

\section{Root Stacks} \label{RootStacks}

\subsection{Definition} \label{rootstackdefn}

Let $X$ be a $\C$-scheme, $L$ an invertible sheaf over $X$, $s \in H^0(X, L)$ and $r \in \N$.  We will define $\XX = \XX_{(L,r,s)}$ to be the category whose objects are quadruples
\begin{align} \label{Xob}
(f : U \to X, N, \phi, t),
\end{align}
where $U$ is a $\C$-scheme, $f$ is a morphism of $\C$-schemes, $N$ is an invertible sheaf on $U$, $t \in H^0(U, N)$ and $\phi : N^{\otimes r} \arsim f^*L$ is an isomorphism of invertible sheaves with $\phi (t^{\otimes r}) = f^* s$.  A morphism
\begin{align*}
(f : U \to X, N, \phi, t) \to (g : V \to X, M, \psi, u)
\end{align*}
consists of a pair $(k, \sigma)$, where $k : U \to V$ is a $\C$-morphism making
\begin{align*}
\commtri{ U }{ V }{ X }{ k }{ f }{ g }
\end{align*}
commute and $\sigma : N \arsim k^* M$ is an isomorphism such that $\sigma(t) = h^*(u)$.  If
\begin{align*}
(g : V \to X, M, \psi, u) \xrightarrow{ (l, \tau) } ( h : W \to X, J, \rho, v)
\end{align*}
is another morphism, then the composition is defined as
\begin{align} \label{rootstackcomposition}
(l, \tau) \circ (k, \sigma) := ( l \circ k, k^* \tau \circ \sigma),
\end{align}
using the canonical isomorphism $(l \circ k)^* J \arsim k^* l^* J$.

We will often use the symbols $\f, \g$ to denote objects of $\XX$.  If we refer to $\f \in \Ob \, \XX(U)$, then it will be understood that this refers to the quadruple $\f = (f : U \to X, N_\f, \phi_\f, t_\f)$.

The category $\XX$ comes with a functor $\XX \to \Sch/\C$ which simply takes $\f$ to the $\C$-scheme $U$ and $(k, \sigma)$ to $k$.

\begin{prop}[{\cite[Theorem 2.3.3]{Cadman}}]
The category $\XX$, together with the structure morphism $\XX \to \Sch/\C$, is a Deligne--Mumford stack.
\end{prop}

There is also a functor $\pi : \XX \to \Sch/X$, whose action on objects and morphisms is given by
\begin{align*}
\f & \mapsto f : U \to X, & (k, \sigma) & \mapsto k;
\end{align*}
this yields a $1$-morphism over $\Sch/\C$, which we will often simply write as $\pi : \XX \to X$.  

\begin{ex}[{\cite[Example 2.4.1]{Cadman}}]\label{affinerootstack}
Suppose $X = \Spec \, A$ is an affine scheme, $L = \OO_X$ is the trivial bundle and $s \in H^0(X, \OO_X) = A$ is a function.  Consider $U = \Spec \, B$ where $B = A[t]/(t^r - s)$.  Then $U$ admits an action of the group (scheme) of $r$th roots of unity $\mmu_r$, where the induced action of $\zeta \in \mmu_r$ is given by
\begin{align*}
\zeta \cdot s & = s, \quad s \in A, & \zeta \cdot t & = \zeta^{-1} t.
\end{align*}
In this case, the root stack $\XX_{(\OO_X, s, r)}$ coincides with the quotient stack $[U/ \mmu_r]$.  Thus, as a quotient by a finite group (scheme), the map $U \to \XX$ is an \'etale cover.
\end{ex}

The root stack $\XX$ comes with a tautological line bundle $\NN$ which can be described as follows.  For an \'etale morphism $\f : U \to \XX$ with $U$ a $\C$-scheme, we define
\begin{align*}
\NN_\f := N_\f,
\end{align*}
where $N_\f$ is the line bundle .  The isomorphisms (\ref{sheafiso}) come from those occurring in the definition of a morphism in the category $\XX$.  One also has a global section $\t \in H^0(\XX, \NN)$ by taking
\begin{align*}
\t_\f := t_\f.
\end{align*}

\subsection{Finite Galois Coverings} \label{finiteGalois}

Let $p : Y \to X$ be a finite Galois covering of smooth quasi-projective 
varieties with Galois group $\Gamma$ and let $\widetilde{D} \subseteq Y$ be the 
locus of points which have non-trivial isotropy.  This is a divisor whose 
irreducible components are smooth \cite[Lemma 2.8]{BiswasOrbifold}.  We will 
assume that all such isotropy subgroups are cyclic of order $r$.  Let $D := 
p(\widetilde{D})$ so that
\begin{align*}
p^* D = r \widetilde{D}\, .
\end{align*}
Then $D$ is a divisor on $X$; let $s \in H^0(X, \OO_X(D))$ be a section defining 
$D$.  Then taking $u \in H^0(Y, \OO_Y(\widetilde{D}))$ defining $\widetilde{D}$, 
adjusting by a non-vanishing function if necessary, we may assume that under the 
isomorphism $\phi_Y : \OO_Y(\widetilde{D})^{\otimes r} \arsim p^* \OO_X(D)$, the 
equation
\begin{align*}
\phi_Y( u^{\otimes r}) = p^*s
\end{align*}
holds.  This defines a morphism $\p : Y \to \XX = \XX_{\OO_X(D), s, r}$.

Let $\f : U \to \XX$ be an arbitrary morphism from a $\C$-scheme $U$, say $\f$ is as in (\ref{Xob}), and consider the fibre product
\begin{align*}
U \times_\XX Y.
\end{align*}

\begin{lem} \label{Gammatorsor}
The projection morphism makes $U \times_\XX Y \to U$ into a $\Gamma$-torsor over $U$.
\end{lem}

\begin{proof}
Since $\XX$ has a representable diagonal, $U \times_\XX Y$ is a scheme (and hence its fibre categories are sets).  The $\Gamma$-action on $U \times_\XX Y$ is induced by that on $Y$.  We need to see that this action is free.  It is enough to check this on the $W$-points $(U \times_\XX Y)(W)$ for an arbitrary $\C$-scheme $W$.  Since $\Gamma$ is a finite group, there is no harm in assuming that $W$ is connected, so that we may identify $\Gamma(W)$ with $\Gamma$ as a group (in the set-theoretic sense).

The fibre category $(U \times_\XX Y)(W)$ is a set whose elements are triples $(a,b, \sigma)$, where $a : W \to U, b : W \to Y$ are morphisms such that
\begin{align*}
\commsq{W}{Y}{U}{\XX}{b}{a}{\hat{\pi}}{\f}
\end{align*}
commutes and $\sigma : a^* N_\f \arsim b^* \OO_Y(\widetilde{D})$ is an 
isomorphism with 
\begin{align} \label{absigma}
\sigma (a^* t_\f) = b^* u.
\end{align}

To explain the $\Gamma$-action on $(U \times_\XX Y)(W)$, we first recall that 
$\Gamma$ acts on the line bundle $\OO_Y(\widetilde{D})$ in a way that is 
compatible 
with the action on $Y$.  This action is via isomorphisms $\OO_Y(\widetilde{D}) 
\arsim \gamma^* \OO_Y(\widetilde{D})$ for each $\gamma \in \Gamma$.  Restricting 
this action to $\widetilde{D}$, we get one on 
$\OO_Y(\widetilde{D})|_{\widetilde{D}}$, 
which 
is the normal bundle to $\widetilde{D}$ in $Y$ (at least away from the 
intersections 
of the irreducible components of $\widetilde{D}$), and the action of $\Gamma$ is 
faithful (see the proof of Lemma 2.8 in \cite{BiswasOrbifold}). 
Now, $\gamma \in \Gamma$ acts on a triple $(a, b, \sigma)$ by taking $b$ to 
$\gamma \circ b$, and acting on $\sigma$ in a way to compensate for the fact that 
in (\ref{absigma}), we will now have $(\gamma \circ b)^* u$ instead of $b^*u$; 
this action comes from that on $\OO_Y(\widetilde{D})$ just described.

Fixing $(a, b, \sigma)$, if $\gamma$ lies in its isotropy subgroup, then it 
follows that $\gamma \circ b = b$, which implies that $b(W) \subseteq 
\widetilde{D}$.  But then $b^* \OO_Y(\widetilde{D})$ is the pullback of the 
normal 
bundle to $\widetilde{D}$ on which $\Gamma$ acts faithfully, so if $\gamma$ also 
fixes $\sigma$, then it must be the identity element.
\end{proof}

\begin{prop} \label{globalquotientstack}
Suppose $p : Y \to X$ is as above.  Then there is an equivalence of stacks
\begin{align*}
\XX \arsim [\Gamma \backslash Y ]\, .
\end{align*}
\end{prop}

\begin{proof}
We define a functor $[\Gamma \backslash Y] \to \XX$.  Suppose we are given an object of $[\Gamma \backslash Y]$ over $U$:  this is a diagram
\begin{align*}
\xymatrix{
P \ar[r]^\sigma \ar[d]_\rho & Y \\ U & }
\end{align*}
where $\rho : P \to U$ is a $\Gamma$-torsor and $\sigma : P \to Y$ is a $\Gamma$-equivariant morphism.  Since $\sigma$ is equivariant, $\sigma^* M$ gives a line bundle on $P$ with a $\Gamma$-action; obviously, it also comes with a section $\sigma^* u$ and an isomorphism $\sigma^* \alpha : (\sigma^* M)^{\otimes r} \arsim (\pi \circ \sigma)^* L$ with $(\sigma^* \alpha)( \sigma^* u)^{\otimes r} = (\pi \circ \sigma)^* s$.

As $\rho : P \to U$ is a $\Gamma$-torsor, $U$ is a geometric quotient by $\Gamma$.  The composition $p \circ \sigma : P \to Y \to X$ is a $\Gamma$-invariant morphism, and hence there is a unique morphism $f : U \to X$ making
\begin{align*}
\commsq{ P }{ Y }{ U }{ X }{ \sigma }{ \rho }{ p }{ f }
\end{align*}
commute.  The fact that $\sigma^* M$ has a compatible $\Gamma$-action means that it and the section $\sigma^* u$ descend to $U$ yielding an object $\f \in \Ob \, \XX(U)$, and a (2-)commutative diagram
\begin{align} \label{torsorsquare}
\vcenter{ \commsq{ P }{ Y }{ U }{ \XX. }{ \sigma }{ \rho }{ p }{ \f } }
\end{align}

Since a morphism in $[\Gamma \backslash Y]$ is simply a pullback diagram of $\Gamma$-torsors, it is clear to see how this functor acts on morphisms.  Thus, we have defined $[\Gamma \backslash Y] \to \XX$.

To go the other way, suppose we are given an object of $\XX$ over $U$, which 
translates into a morphism $\f : U \to \XX$.  Then by Lemma \ref{Gammatorsor}, the top and left arrows of the Cartesian square 
\begin{align} \label{Yfibre}
\vcenter{
\commsq{ U \times_\XX Y }{ Y }{ U }{ \XX, }{}{}{p}{\f} }
\end{align}
yield an object of $[ \Gamma \backslash Y]$.  A morphism in (the category) $\XX$ translates to a ($2$-)commutative diagram
\begin{align*}
\commtri{ U }{ V }{ \XX }{}{}{}
\end{align*}
in which case, one sees that the appropriate definition of this functor on morphisms is to take the pullback diagram of the torsor obtained above.  

One will note that the square in (\ref{torsorsquare}) is, in fact, Cartesian.  
Lemma \ref{Gammatorsor} states that the fibre product $U \times_\XX Y$ is a $\Gamma$-torsor over $U$.  The commutativity of (\ref{torsorsquare}) yields a morphism $P \to U \times_\XX Y$ which will be a morphism of $\Gamma$-torsors over $U$ and hence an isomorphism.  Once account is taken of this, one realizes that the diagrams (\ref{torsorsquare}) and (\ref{Yfibre}) are essentially the same, and that the functors are quasi-inverses of each other.  
\end{proof}

\subsection{Root Stacks Over Curves}

The following statement (and its attribution) can be found in \cite[Theorem 1.2.15]{Namba}.

\begin{theorem}[Bundgaard--Nielsen--Fox]\label{BNF}
Let $X$ be an irreducible projective algebraic curve over $\C$.  Let $Z := \{ 
x_1, \cdots, x_m \} \subseteq X$ be set of $m$ distinct points and let $r_1, 
\cdots, r_m \in \N_{\geq 2}$.  If $g = 0$, then we will assume that either $m 
\geq 3$ or $m = 2$ with $r_1 = r_2$.  Then there exists a projective curve $Y$ 
and 
a Galois covering $p : Y \to X$ such that $p|_{Y \setminus p^{-1}(Z)}$ is unramified and if $y \in p^{-1}(x_i)$ then the ramification index of $y$ is precisely $r_i$.
\end{theorem}

By taking $r := r_1 = \cdots = r_m$, we obtain the following from Proposition \ref{globalquotientstack}.

\begin{cor} \label{globalquotient}
Let $X$ be as in the Theorem.  If $g \geq 1$ or $m > 1$, then the associated root stack $\XX = \XX_{\OO_X(S), s, r}$ is a global quotient stack.  Precisely, if a Galois covering $p : Y \to X$ is chosen as in the Theorem, with Galois group $\Gamma$, then there is an equivalence
\begin{align*}
\XX \arsim [\Gamma \backslash Y ].
\end{align*}
\end{cor}

\section{Bundles and Root Stacks} \label{BundlesRootStacks}

\subsection{Parabolic Vector Bundles and Root Stacks}

We recall the definition of a parabolic vector bundle over a $\C$-scheme $X$ with respect to an effective Cartier divisor $D$ given in \cite[\S1]{Yokogawa1995}.  We will use $\R$ as an index category, whose objects are real numbers and in which a (single) morphism $\beta \to \alpha$ exists, by definition, precisely when $\beta \leq \alpha$.  Let $\E_*$ be a functor $E_* : \R^\op \to \QCoh(X)$, where $\QCoh(X)$ is the category of quasi-coherent $\OO_X$-modules.  If $\alpha \in \R$, we simply write $E_\alpha$ for $E_*(\alpha)$, and $i_{\alpha, \beta}$ for the morphism $E_\alpha \to E_\beta$ given by the functor $E_*$ when $\alpha \geq \beta$.  

Given $E_*$ as above and $\gamma \in \R$, one can define another functor $E[\gamma]_* : \R^\op \to \QCoh(X)$  by setting
\begin{align*}
E[\gamma]_\alpha := E_{\alpha + \gamma},
\end{align*}
together with the only possible definition on morphisms.  If $\gamma \geq 0$, then there is a natural transformation $E[\gamma]_* \to E_*$.  The functor $E_*$ is called a \emph{parabolic sheaf} if it comes with a natural isomorphism of functors $j : E_* \otimes_{\OO_X} \OO_X(-D) \arsim E[1]_*$ fitting into a commutative diagram
\begin{align*}
\xymatrixcolsep{5pc}
\xymatrix{
E_* \otimes_{\OO_X} \OO_X(-D) \ar[r]^-j \ar[rd] & E[1]_* \ar[d] \\
& E_*}
\end{align*}
The sheaf $E_0$ is often referred to as the \emph{underlying sheaf} and written as simply $E$.  A morphism of parabolic sheaves $(E_*, j) \to (F_*, k)$ is a natural transformation $E_* \to F_*$ intertwining $j$ and $k$.

A parabolic sheaf $(E_*, j)$ is said to be \emph{coherent} if it factors through $\Coh(X) \to \QCoh(X)$, where $\Coh(X)$ is the category of $\OO_X$-modules and further if there is a finite sequence of real numbers $0 \leq \alpha_1 < \alpha_2 < \cdots < \alpha_l < 1$ such that if $\alpha \in (\alpha_{i-1}, \alpha_i]$, then
\begin{align*}
i_{\alpha_i, \alpha} : E_{\alpha_i} \to E_\alpha
\end{align*}
is the identity map.  We will thus have a filtration of sheaves
\begin{align} \label{sheaffiltration}
E := E_0 \supset E_{\alpha_1} \supset E_{\alpha_2} \supset \cdots \supset E_{\alpha_l} \supset E_1 \arsim E(-D).
\end{align}
A coherent parabolic sheaf $(E_*, j)$ is called a \emph{parabolic vector bundle} if $E_*$ takes values in the category $\Vect(X)$ of vector bundles over $X$ and further, whenever $\beta \leq \alpha < \beta +1$, the sheaf $\coker \, i_{\alpha, \beta}$, which is supported on $D$, is locally free as an $\OO_D$-module.  The category of parabolic vector bundles will be denoted by $\ParVect_D(X) = \ParVect(X)$.

We will say that the coherent parabolic sheaf $(E_*, j)$ has \emph{rational weights} if the $\alpha_i, 1 \leq i \leq l,$ may be chosen in $\Q$.  Since this is a finite set, these weights may be chosen in $\frac{1}{r}\Z$ for some $r \in \N$; in this case, we will think of $E_*$ as a functor $(\frac{1}{r} \Z)^\op \to \Coh(X)$.  For obvious reasons, we can say then that $E_*$ has \emph{weights dividing $r$}.  We will denote by $\ParVect_{D,r}(X) = \ParVect_r(X)$ the category of parabolic vector bundles with weights dividing $r$.

One of the main results of \cite{Borne} is that parabolic vector bundles on $X$ correspond to vector bundles on an appropriate root stack.  We give a precise statement.  With $X$ and $D$ as above.  Let $s \in H^0(X, \OO_X(D))$ be a section defining the divisor $D$ and fix $r \in \N$.  We will let $\XX := \XX_{\OO_X(D), s, r}$ be the corresponding root stack.

\begin{theorem}[{\cite[Th\'eor\`eme 3.13]{Borne}}]\label{BorneVB}
The functor $\Vect(\XX) \to \ParVect_{D,r}(X)$ which takes $\V$ to the functor $E_* : (\frac{1}{r}\Z)^\op \to \Vect(X)$
\begin{align*}
\tfrac{i}{r} \mapsto \pi_* ( \V \otimes_{\OO_\XX} \NN^{\otimes -i} )
\end{align*}
is an equivalence of tensor categories.
\end{theorem}

\subsection{Degree of a Vector Bundle over a Root Stack}

Assume now that $X$ is a smooth projective variety with very ample invertible sheaf $\OO_X(1)$.  Let $D, r,$ and $\XX$ be as above.  Then if $\V$ is a vector bundle over $\XX$, \cite[D\'efinition 4.2]{Borne} defines its \emph{degree} as
\begin{align*}
\deg_\XX \V := q_* \big( c_1^{et}(\V) \cdot \pi^* c_1^{et} \big( \OO_X(1) \big)^{n-1} \big),
\end{align*}
where $q : \XX \to \Spec \, \C$ is the structure morphism.

One has the following theorem.

\begin{theorem}[{\cite[Th\'eor\`eme 4.3]{Borne}}]\label{Bornedeg}
Let $\V$ be a vector bundle on the root stack $\XX$ and $E_*$ the corresponding parabolic vector bundle over $X$ (via the equivalence given in Theorem \ref{BorneVB}).  Then
\begin{align*}
\pardeg \, E_* = \deg_\XX \V.
\end{align*}
\end{theorem}

\subsection{Principal Bundles Over Root Stacks} \label{PBRS}

We return to the situation of Section \ref{finiteGalois}, where  $p : Y \to X$ is a finite Galois covering of smooth quasi-projective varieties with Galois group $\Gamma$.  By Proposition \ref{stackGamma}, we immediately obtain the following.

\begin{cor} \label{GammaGcurve}
With $\XX$ as in Corollary \ref{globalquotient}, there is an equivalence
\begin{align*}
\Bun_G \, \XX \arsim \Bun_{\Gamma, G} Y.
\end{align*}
\end{cor}

We will now restrict to the case where $X$ is a smooth projective curve, where we can give something of a refinement of this equivalence.   Let $x \in X, y \in Y$, and let $\OO_x, \OO_y$ be the respective local rings, and $\widehat{\OO}_x, \widehat{\OO}_y$ their completions, and $\K_x, \K_y$ the respective quotient fields.  As a matter of notation, we will set
\begin{align*}
\DD_x & := \Spec \, \widehat{\OO}_x, & \DD_x^\times & := \Spec \, \K_x, & \DD_y & := \Spec \, \widehat{\OO}_y, & \DD_y^\times & := \Spec \, \K_y.
\end{align*}
One has a Cartesian diagram
\begin{align} \label{Xx}
\vcenter{ \commsq{ \DD_x^\times }{ X \setminus x }{ \DD_x }{ X. }{}{}{}{} }
\end{align}

The local type of a $(\Gamma, G)$-bundle is defined in \cite[\S2.2]{BalajiSeshadri2012} as follows.  Let $E$ be a $(\Gamma, G)$-bundle over $Y$.  Then for each $y \in p^{-1}(Z)$, $E|_{\DD_y}$ is a $(\Gamma_y, G)$-bundle and this is defined by a representation $\rho_y : \Gamma_y \to G$.  Let $\tau_y$ denote the equivalence class of the representation $\rho_y$.  Then the \emph{local type of $E$} is defined as
\begin{align*}
\tau := \{ \tau_y \, : \, y \in p^{-1}(Z) \}.
\end{align*}
We let $\Bun_{\Gamma, G}^\tau Y$ denote the stack of $(\Gamma, G)$-bundles of local type $\tau$.

Because of Corollary \ref{GammaGcurve}, there should be a well-defined notion of a local type for a $G$-bundle over $\XX$.  Fix $x \in Z$ and let $y_1, y_2 \in p^{-1}(x)$.  Then there exists $\gamma \in \Gamma$ such that conjugation by $\gamma$ yields an isomorphism $\Gamma_{y_1} \arsim \Gamma_{y_2}$; since $\Gamma_{y_1}, \Gamma_{y_2}$ are abelian, this isomorphism is in fact independent of the choice of $\gamma$.  Thus, it is possible to choose isomorphisms $c_y : \mmu_r \arsim \Gamma_y$ for each $y \in p^{-1}(x)$ such that for all $y_1, y_2 \in p^{-1}(x)$, the diagram
\begin{align*}
\ucommtri{ \mmu_r }{ \Gamma_{y_1} }{ \Gamma_{y_2} }{ c_{y_1} }{ c_{y_2} }{}
\end{align*}
commutes.  This makes each $E|_{\DD_y}$ a $(\mmu_r, G)$-bundle and if $\gamma$ is as above, it also yields an isomorphism of $E|_{\DD_{y_1}} \arsim E|_{\DD_{y_2}}$ as $(\mmu_r, G)$-bundles, and hence the local representations $\rho_{y_1} \circ c_{y_1}, \rho_{y_2} \circ c_{y_2} : \mmu_r \to G$ are equivalent.

Observe that $\DD_x \times_X Y \cong \coprod_{y \in p^{-1}(x)} \DD_y$, so it follows that
\begin{align*}
\DD_x \times_X \XX \cong \DD_x \times_X [\Gamma \backslash Y] \cong [ \Gamma \backslash \DD_x \times_X Y] \cong \left[\Gamma \bigg\backslash \coprod_{y \in p^{-1}(x)} \DD_y \right] \cong [ \Gamma_y \backslash \DD_y ] \cong [ \mmu_r \backslash \DD_y ],
\end{align*}
where, in the last two expressions, $y \in p^{-1}(x)$ is any choice of preimage.

Now, given a $G$-bundle $\E$ on $\XX$, $\E|_{\DD_x \times_X \XX}$ is a $(\mmu_r, G)$-bundle over $\DD_y$.  This may be identified with the restriction of the associated $(\Gamma, G)$-bundle on $Y$ restricted to $\DD_y$ for some $y \in p^{-1}(x)$.  There is thus a well-defined equivalence class of a homomorphism $\mmu_r \to G$, which we denote by $\tau_x$ and call the \emph{local type of $\E$ at $x$}.  We define the \emph{local type of $\E$} to be $\tau := \{ \tau_x \, : \, x \in Z \}$.  We will denote by $\Bun_G^\tau \XX$ the stack of $G$-bundle over $\XX$ of local type $\tau$.

\begin{prop} \label{GammaGlocaltype}
The equivalence of Corollary \ref{GammaGcurve} restricts to equivalences
\begin{align*}
\Bun_G^\tau \XX \arsim \Bun_{\Gamma, G}^\tau Y
\end{align*}
for each local type $\tau$.
\end{prop}

\begin{rmk}
The local type of a $G$-bundle $\E$ on $\XX$ is independent of our realization of $\XX$ as a quotient stack $[\Gamma \backslash Y]$.  If we fix $x \in Z$ and set $B := \widehat{\OO}_x [t]/(t^r - z) \cong \C[\![ t ]\!]$, where $z \in \widehat{\OO}_x$ is a parameter at $x$, we observe that $B$ admits a $\mmu_r$-action for which $\widehat{\OO}_x$ is the ring of invariants and if $\DD_{\hat{x}} := \Spec \, B$, then we have an abstract isomorphism
\begin{align*}
\DD_x \times_X \XX \cong [\mmu_r \backslash \DD_{\hat{x}} ].
\end{align*}
Then a $\E$ restricts to a $G$-bundle on $\DD_x \times_X \XX$ and hence corresponds to a $(\mmu_r, G)$-bundle on $\DD_{\hat{x}}$, which determines the local type.
\end{rmk}

\section{Connections On Vector Bundles over a Root Stack} \label{Connections}

In this section, we will assume $X$ to be a smooth irreducible curve over $\C$.  Let $Z \subseteq X$ be a reduced divisor.  Let $s \in H^0(X, \OO_X(Z))$ be a section defining $Z$ and fix $r \in \N$.  We will let $\XX = \XX_{ \OO_X(Z), s, r }$ be the associated root stack.

\subsection{Parabolic Connections}

Suppose $E_*$ is a rank $n$ parabolic vector bundle over $X$ given as a filtered sheaf as in (\ref{sheaffiltration}).  It is easy to recover the parabolic structure on the underlying vector bundle $E$ in the original sense of \cite{MS} in terms of a weighted flag in the fibre $E_x$ for each $x \in \supp \, Z$.  Given the filtration (\ref{sheaffiltration}), we take the images of the fibres of the $E_i$ in $E_x$ to get a flag
\begin{align} \label{flag}
E_x = E_{x,1} \supset E_{x,2} \supset \cdots \supset E_{x,k} \supset E_{x,k+1} = \{ 0 \},
\end{align}
and the weight $\alpha_i$ attached to $E_{x,i}$ is the largest $\alpha$ such that $E_{x,i} = i_{\alpha,0}( (E_\alpha)_x)$.

Let $D$ be a connection on $E$ with (logarithmic) simple poles at $Z$.  If $x \in \supp \, Z$ then the residue $\Res_x D$ is a well-defined endomorphism of $E_x$.  We say that $D$ is a \emph{parabolic connection} if for each $x \in \supp \, Z$, 
\begin{align} \label{parcxn}
\Res_x D (E_i) & \subseteq E_{i+1}, & \text{ and } & & \Res_x D|_{E_i/E_{i+1}} & = \alpha_i \1_{E_i/E_{i+1}},
\end{align}
for $1 \leq i \leq k$ \cite[\S2.2]{BiswasLogares2011}.

We will use the following \cite[Lemma 4.2]{BiswasLogares2011}.

\begin{lem} \label{parlineconnection}
A parabolic line bundle $L_*$ admits a connection if and only if $\pardeg \, L_* = 0$.
\end{lem}

\subsection{Connections on the Root Stack and Parabolic Connections}

\begin{prop} \label{rootstackconnection}
The category of vector bundles with connections on $\XX$ is equivalent to the category of parabolic vector bundles with parabolic connections on $X$.
\end{prop}

\begin{proof}
Suppose we are given a rank $n$ vector bundle and connection $(\V, \nabla)$ on $\XX$.  We want to show that $\nabla$ induces a parabolic connection on the corresponding parabolic vector bundle $E_*$ on $X$.  Since a connection is defined locally and the parabolicity condition (\ref{parcxn}) is also local, as in Example \ref{affinerootstack}, we may assume that $X = \Spec \, A$, that $\supp \, Z = \{ x \}$ is a single parabolic point defined by $s \in A$ whose image in $\OO_{X,x}$ is a parameter at $x$ and such that $ds$ is a local basis for $\Omega_{X/\C}^1$, so that if $B := A[t]/(t^r - s)$ and if $U := \Spec \, B$, then $\XX = [U/ \mmu_r]$.  Note also that $\Omega_{U/\C}^1$ has $dt$ as a local basis.  If $\gamma \in \mmu_r$ is a generator, we will assume that $\gamma \cdot t = \zeta^{-1} t$ and similarly $\gamma \cdot dt = \zeta^{-1} \, dt$.

In this case, $\V$ is defined by a projective module over $B$ with a compatible 
$\mmu_r$-action, and $\nabla$ commutes with this action.  By shrinking as necessary, we may assume that the module is free over $U$, say with basis $e = \{ e_1, \cdots, e_n \}$, and the $\mmu_r$-action is appropriately diagonalized \cite[Proposition 3.15]{Borne} so that 
\begin{align*}
\gamma \cdot e_j = \zeta^{p_j} e_j,
\end{align*}
for some $p_j$ which satisfy $0 \leq p_n \leq \cdots \leq p_1 \leq r-1$.  Take $1 
\leq j_k < j_{k-1} < \cdots < j_1 = n$ such that if $ j_{i+1} + 1 \leq j \leq j_i$, then $p_j = p_{j_i}$; set $m_i := p_{j_i}$ and $\alpha_i := m_i/r$.  The $\mmu_r$-invariants of the submodule generated by $e_j$ is generated by $t^{p_j} e_j$.  Hence $\pi_* \V$ has a basis $f = \{ f_1 := t^{p_1} e_1, \cdots, f_n := t^{p_n} e_n \}$ or
\begin{align*}
& f_1 = t^{m_k} e_1, \cdots, f_{j_k} = t^{m_k} e_{j_k}, f_{j_k + 1} = t^{m_{k-1}} 
e_{j_k + 1}, \cdots, f_{j_{k-1}} = t^{m_{k-1}} e_{j_{k-1}}, \cdots, \\
& f_{j_2 + 1} = t^{m_1} e_{j_2 + 1}, \cdots, f_n = t^{m_1} e_n.
\end{align*}
More generally, if $m_{i-1} + 1 \leq l \leq m_i$, then $\pi_*( \V \otimes_{\OO_\XX} \NN^{\otimes -l})$ is spanned by
\begin{align*}
f_1, \cdots, f_{j_i}, s f_{j_i+1}, \cdots, s f_n.
\end{align*}
Then the subspace $V_i$ of $V_x$ is spanned by $f_1(x), \cdots, f_{j_i}(x)$.  
This describes the filtration (\ref{flag}).

Now, suppose that, with respect to the basis $e$, $\nabla$ has the connection matrix $\omega = (\omega_{ij}) dt$, so that
\begin{align*}
\nabla e_j = \sum_{i=1}^n \omega_{ij} e_i \, dt.
\end{align*}
Then comparing the two expressions
\begin{align*}
\nabla \gamma \cdot e_j & = \zeta^{p_j} \sum_{i=1}^n \omega_{ij} e_i \, dt, & \gamma \cdot \nabla e_j & = \sum_{i=1}^n \zeta^{p_i -1} (\gamma \cdot \omega_{ij}) e_i \, dt,
\end{align*}
we find that $\gamma \cdot \omega_{ij} = \zeta^{p_j - p_i + 1} \omega_{ij}$.  Hence $\omega_{ij}$ is of the form
\begin{align*}
\omega_{ij} = \begin{cases}  t^{p_i - p_j - 1} \nu_{ij} & \text{ if } p_i > p_j \\ st^{p_i - p_j - 1} \nu_{ij} & \text{ if } p_i \leq p_j, \end{cases}
\end{align*}
for some $\nu_{ij} \in A$.

The change of basis matrix, from $e$ to $f$, is $g = \diag( t^{p_1}, \cdots, 
t^{p_n})$ and so the connection matrix with respect to $f$ is $g^{-1} \omega g + g^{-1} \, dg$, the $(i,j)$-entry of which is 
\begin{align*}
\begin{cases}  \nu_{ij} \frac{dt}{t} = \frac{1}{r} \nu_{ij} \frac{ds}{s} & \text{ if } p_i > p_j \\ \nu_{ij} \frac{ s \, dt }{t} + \delta_{ij} p_i \frac{dt}{t} = \frac{1}{r} ( s \nu_{ij} + \delta_{ij} p_i) \frac{ds}{s} & \text{ if } p_i \leq p_j. \end{cases}
\end{align*}
One sees immediately that this gives a well-defined logarithmic connection $D$ on $E_*$ and yields the following expression for the residue at $x$
\begin{align*}
(\Res_x D)_{ij} = \begin{cases}  \frac{1}{r} \nu_{ij}(x) & \text{ if } p_i > p_j \\  \delta_{ij} \frac{p_i}{r} & \text{ if } p_i \leq p_j. \end{cases}
\end{align*}
From this, it is straightforward to verify that $D$ is, in fact, a parabolic connection.

In the other direction, suppose we are given a parabolic connection $(E_*, D)$.  
Let $f_1, \cdots, f_n$ be a local frame and suppose $V_x$ has a flag (\ref{flag}) with weights $\alpha_i = m_i/r, 1 \leq i \leq k$.  Then the corresponding bundle on $\XX$ is represented over $U$ by the free $B$-module with basis
\begin{align*}
& e_1 := f_1 \otimes t^{-m_k}, \cdots, e_{j_k} := f_{j_k} \otimes t^{-m_k}, 
e_{j_k + 1} := f_{j_k + 1} \otimes t^{-m_{k-1}}, \cdots, e_{j_2} := f_{j_2} 
\otimes t^{-m_2}, \\
& e_{j_2 + 1} := f_{j_2 + 1} \otimes t^{-m_1}, \cdots, e_n := f_n \otimes 
t^{-m_1}.  
\end{align*}
Now, reversing the argument above, we see without difficulty that the induced connection on the $B$-module has no poles and is compatible with the $\mmu_r$-action.  Hence we obtain a connection on $\XX$.

Observe that in the above how we go from a connection on $\XX$ to a parabolic connection and back is essentially via a ``change of basis'' operation.  When realized as such, it is clear that the two operations are inverse to each other.

A morphism $\varphi : (\V, \nabla^\V) \to (\WW, \nabla^\WW)$ is a morphism $\varphi : \V \to \WW$ which commutes with the respective connections, i.e.\ $\nabla^\WW \circ \varphi = ( \varphi \otimes \1_{\Omega_\XX^1}) \circ \nabla^\V$.  In local frames, this means that the matrix for $\varphi$ satisfies the same relation with the respective connection matrices.  Let $(E_*, D^E), (F_*, D^F)$ be the corresponding parabolic connections.  Then, as we saw, the connection matrices for $D^E$ and $D^F$ are obtained by changes of basis on each of $\V$ and $\WW$ from the matrices for $\nabla^\V$ and $\nabla^\WW$; the matrix for $\pi_* \varphi$ will be obtained from these same changes of basis, so the commutation property will be preserved and we get a morphism of connections.  Again, the process is reversible.
\end{proof}

The following is immediate from the Proposition, Lemma \ref{parlineconnection} and Theorem \ref{Bornedeg}.

\begin{cor} \label{linecxnroot}
A line bundle $\mathcal{L}$ on $\XX$ admits a connection if and only if $\deg_\XX \, \mathcal{L} = 0$.
\end{cor}

\section{A Condition for the Existence of a Connection} \label{ConnectionCondition}

In this section $X$ will be an irreducible smooth complex projective curve, $Z 
\subseteq X$ a reduced divisor, $s \in H^0(X, \OO_X(Z))$ a section defining $Z$, 
$r \in \N$ and $\XX = \XX_{\OO_X(Z), s, r}$ the associated root stack.  We will 
further assume that either $g \geq 1$ or $m > 1$, so as to be able to apply 
Corollary \ref{globalquotient}.  As before, $G$ is a reductive complex algebraic 
group.

The following theorem is a generalization of a criterion of Weil and Atiyah,
\cite{At}, \cite{We}, for
the existence of a holomorphic connection on a holomorphic vector bundle over
a compact Riemann surface

\begin{theorem} \label{connectioncondition}
A principal $G$-bundle $\E$ on $\XX$ admits a connection if and only if for any reduction to a Levi factor $L$ of a parabolic subgroup of $G$, say $\F$ is an $L$-bundle with $\F \times^\iota G \cong \E$, and any character $\chi : L \to \C^\times$, the associated line bundle $\mathcal{M} = \mathcal{M}^\chi := \F \times^\chi \C$ satisfies
\begin{align*}
\deg_\XX \mathcal{M} = 0.
\end{align*}
\end{theorem}

\begin{proof}
Suppose first that $\E$ admits a connection and let $\F$ be a reduction of $G$ to a Levi subgroup $L$ and $\chi : L \to \C^\times$ a character.  Then by Lemmata \ref{Levisplitting}, \ref{splittingconnection} and \ref{inducedconnection}, $\mathcal{M}^\chi$ on $\XX$ admits a connection.  Hence $\deg_\XX \, \mathcal{M}^\chi = 0$ by Corollary \ref{linecxnroot}.

We now prove the converse.  Assume that $\E$ satisfies the condition of the Theorem.  We choose a Galois cover $p : Y \to X$ as in Corollary \ref{globalquotient}.  Then we obtain a surjective \'etale morphism $\p : Y \to \XX$, so that $Y$ is an atlas for $\XX$.  Corresponding to the morphism $\p$ is a $G$-bundle $\E_\p$ which, since $\XX = [\Gamma \backslash Y]$, comes with a compatible $\Gamma$-action.  A reduction of the structure group $\E$ to a Levi subgroup $H$ corresponds to a $\Gamma$-invariant reduction of $\E_\p$ to $H$, and conversely, so the hypotheses of the following Lemma are satisfied.

\begin{lem}
Let $E$ be a $\Gamma$-linearized principal $G$-bundle over $Y$ such that for every Levi subgroup $H \subseteq G$, every $\Gamma$-invariant holomorphic reduction $F$ of $E$ to $H$ and every character $\chi : H \to \C^\times$, one has
\begin{align*}
\deg F \times^\chi \C = 0.
\end{align*}
Then $E$ admits a $\Gamma$-invariant connection.
\end{lem}

Except for the $\Gamma$-invariance of the connection, this is the statement of \cite[Lemma 4.2]{BiswasPPB}; that the existence of one connection implies the existence of a $\Gamma$-invariant one is proved by an averaging argument on the previous page of the same paper.  

Now, the existence of a $\Gamma$-invariant connection on $\E_\p$ implies that there is a $\Gamma$-invariant splitting of the Atiyah sequence for $\E_\p$.  Since the question is now framed in terms of the existence of a section of an appropriate sheaf over $Y$, such a splitting descends to a splitting of the Atiyah sequence for $\E$ and we conclude by Lemma \ref{Atiyahcondition}.
\end{proof}

Let $\E$ be a principal $G$-bundle over $\XX$. Consider the short exact sequence
$$
0 \to \Omega^1_{\XX}\otimes \ad \, \E \to \Omega^1_{\XX}\otimes \At \, \E
\stackrel{\mu}{\to} \Omega^1_{\XX}\otimes T \XX = End(T \XX) \to 0
$$
obtained by tensoring the Atiyah sequence associated to
$\E$ with $\Omega^1_{\XX}$. It produces the exact sequence
$$
0 \to \Omega^1_{\XX}\otimes \ad \, \E \to \mu^{-1}(\text{Id}_{T \XX}\cdot
\OO_{\XX}) \stackrel{\mu}{\to} \text{Id}_{T \XX}\cdot
\OO_{\XX} \,=\, \OO_{\XX} \to 0\, .
$$
There is a natural bijective correspondence between the
splitting of this exact sequence and the above Atiyah sequence associated to
$\E$. In \cite{BiswasPPB}, connections on parabolic principal bundles were
defined to be the splittings of the exact sequence in the
parabolic context given by this exact sequence.

\section{Tensor Functors} \label{TensorFunctors}

Let $X$ be a scheme over an arbitrary field $k$, $G$ an affine group scheme over 
$k$ and $E$ a $G$-bundle over $X$.  Then the assignment taking a representation 
$\rho : G \to {\rm GL}(V)$ to the associated vector bundle $E \times^\rho V$ 
defines 
a functor $F_E : \Gmod \to \Vect \, X$, where $\Gmod$ is the category of 
finite-dimensional representations of $G$ and one observes that $F_E$ satisfies 
the following properties:
\begin{enumerate}
\item[(i)] $F_E$ is $k$-additive;
\item[(ii)] $F_E$ is a tensor functor in in the sense that it commutes with the 
formation of tensor products (in each of the respective categories), and with 
the natural isomorphisms of functors which give the associativity and 
commutativity of the tensor product in each category;
\item[(iii)] $F_E$ takes the trivial $1$-dimensional representation to the 
trivial line bundle;
\item[(iv)] $F_E$ takes an $n$-dimensional representation to a rank $n$ vector 
bundle.
\end{enumerate}
Following a Tannakian philosophy, M.V.\ Nori was able to see that any functor $\Gmod \to \Vect \, X$ satisfying these conditions in fact comes from a $G$-bundle over $X$ \cite[\S2]{Nori1976}.

The approach to generalizing the notion of a parabolic vector bundle to that of a parabolic principal bundle taken by \cite{BBN} is to view a $G$-bundle in this sense.  Thus, one defines a parabolic principal bundle as a functor $\Gmod \to \ParVect_{D,r}(X)$, for some $r \in \N$, which satisfies conditions (i)-(iv) above \cite[Definition 2.5]{BBN}.  One will recall that the original definition \cite{MS} of a parabolic structure (on a vector bundle over a smooth curve) consisted of a flag of subspaces of the fibre over a parabolic point $x$, together with a set of weights in $[0,1)$.  While a flag has a clear $G$-bundle analogue in terms of an element in a generalized flag manifold, it was much less obvious what the correct generalization for a set of weights should be.
The definition in \cite{BBN} was meant to overcome this difficulty.

We now see that the tensor functor approach to parabolic principal bundles coincides with our approach via root stacks.

\begin{prop}
The category of $G$-bundles on the root stack $\XX$ is equivalent to the category of tensor functors $\Gmod \to \ParVect_{D,r}(X)$.
\end{prop}

\begin{proof}
The first thing to recall is that the equivalence of $\Vect \, \XX$ and $\ParVect_{D,r} \, X$ is a tensor functor, so satisfies (ii)-(iv); it obviously satisfies (i), and it is clear that it preserves rank and that the trivial bundle on $\XX$ corresponds to the trivial parabolic vector bundle on $X$.  We will also note that if $\f : U \to \XX$ is an \'etale morphism, then the functor $R_\f : \Vect \, \XX \to \Vect \, U$ given by
\begin{align*}
\V \mapsto \V_\f
\end{align*}
also has the same properties, virtually by definition.

Suppose we are given a principal bundle $\E$ on the root stack $\XX$.  Then the associated vector bundle construction of Section \ref{associatedbundles} will give a functor $F_\E : \Gmod \to \Vect \, \XX$ satisfying the conditions above.  Composing with the equivalence $\Vect \, \XX \arsim \ParVect_{D,r}(\XX)$ give a parabolic principal bundle in the sense \cite{BBN}.

Suppose we are given a functor $F : \Gmod \to \ParVect_{D,r} X$ satisfying (i)-(vi) above.  Then given an \'etale morphism $\f : U \to \XX$, we may consider the composition
\begin{align*}
\Gmod \xrightarrow{F} \ParVect_{D,r} \, X \arsim \Vect \, \XX \xrightarrow{R_\f} \Vect \, U,
\end{align*}
which we will denote by $F_\f$.  This will satisfy (i)-(vi) and hence determines a principal bundle $\E_\F$ over $U$.  Now, given a diagram (\ref{overXX}), one obtains a diagram of categories and functors
\begin{align*}
\rcommtri{ \Gmod }{ \Vect \, V }{ \Vect \, U }{ F_\g }{ F_\f }{ k^* }
\end{align*}
In this situation, \cite[Proposition 2.9(a)]{Nori1976} provides canonical isomorphisms $\E_f \arsim k^* \E_\g$.  Because of the canonical nature of these isomorphisms, the compatibility condition (\ref{compatibilitycondition}) is satisfied.  Thus $F$ defines a principal $G$-bundle over $\XX$.  It is clear that these constructions are inverses of each other, as it is a question of seeing that this is the case at each $\f \in \Ob \, \XX$.
\end{proof}

\section{Parahoric Torsors} \label{ParahoricTorsors}

\subsection{Parahoric Subgroups}

We will assume that $G$ is semisimple and fix a maximal torus $T$ and a Borel subgroup $B$ containing $T$.  Let $A := \C[\![ z ]\!]$ be the ring of formal power series and $K := \C (\!( z )\!) = A[z^{-1}]$ its quotient field (the ring of formal Laurent series).  One has a quotient map $A \to \C$, which yields an evaluation map $\ev : G(A) \to G(\C)$.  The \emph{Iwahori subgroup} $\mathcal{I}$ of $G(K)$ is defined to be
\begin{align*}
\mathcal{I} := \ev^{-1}(B) = \ev^{-1}\big(B(\C) \big).
\end{align*}
A \emph{parahoric subgroup} $\PPP$ of $G(K)$ is one which contains a $G(K)$-conjugate of $\mathcal{I}$.  It is a theorem of \cite{BruhatTitsII} that any such subgroup $\PPP$ is the group of $A$-points for a uniquely defined smooth group scheme over $A$ which, at the risk of poor notation, we will also call $\PPP$.

Let $\g$ denote the Lie algebra of $G$, let $\Phi$ be the root system for $G$ (with respect to $T$) and for $\alpha \in \Phi$, let $\g_\alpha$ denote the corresponding root space and $U_\alpha \subseteq G$ the root group.  We will fix non-zero $x_\alpha \in \g_\alpha$.  

Let $\theta \in \t_\R = Y(T) \otimes_\Z \R$ and consider the subgroup
\begin{align} \label{Ptheta}
\PPP_\theta := \langle T(A), U_\alpha( z^{m_\alpha(\theta)} A ) \, : \, \alpha \in \Phi \rangle.
\end{align}
where $m_\alpha(\theta) = -\lfloor \alpha(\theta) \rfloor$ when we consider $\alpha$ as an element of $\t_\R^*$.  Such a subgroup is parahoric in the above sense and any parahoric subgroup containing $\mathcal{I}$ is of the form $\PPP_\theta$ for some $\theta$ (though such a $\theta$ is clearly not unique).  In what follows, we will typically take $\theta \in Y(T) \otimes_\Z \Q$ or, when $r \in \N$ is fixed, in $Y(T) \otimes_\Z \tfrac{1}{r} \Z$.

For such a $\theta$, \cite[\S2.2]{Boalch_Parahoric} gives the following description of the associated parahoric Lie algebra.  For $\lambda \in \R$, we will set
\begin{align*}
\g_\lambda^\theta := \{ \xi \in \g \, : \, [\theta, \xi] = \lambda \xi \}.
\end{align*} 
Note that $\g_0^\theta \subseteq \g$ is the centralizer subalgebra of $\theta$ and that $\t \subseteq \g_0^\theta$.
For $i \in \Z$, we set 
\begin{align*}
\g^\theta(i) \,:=\, 
\sum_{\lambda \geq -i } \g_\lambda^\theta
\end{align*}
Then we define
\begin{align*}
\wp_\theta := \left\{ \sum_{i \in \Z} \xi_i z^i \in \g(K) \, : \, \xi_i \in \g^\theta(i) \right\}.
\end{align*}
It is not hard to see that this is the same as
\begin{align*}
\wp_\theta = \t(A) \oplus \sum_{\alpha \in \Phi} \g_\alpha( z^{m_\alpha(\theta)} A ),
\end{align*}
which is what one would expect from the description in (\ref{Ptheta}).

An alternative description given in \cite{Boalch_Parahoric} is to consider the (finite-dimensional) vector spaces
\begin{align*}
\g(K)_\lambda^\theta := \left\{ \sum_{i \in \Z} \xi_i z^i \in \g(K) \, : \, \xi_i \in \g_{\lambda-i}^\theta \right\},
\end{align*}
for $\lambda \in \R$, which we will call the \emph{weight $\lambda$} piece of $\g(K)$.  Then $\wp_\theta$ is the sub-algebra with weights $\geq 0$.  In particular, the weight zero piece
\begin{align*}
\g(K)_0^\theta := \left\{ \sum_{i \in \Z} \xi_i z^i \in \g(K) \, : \, \xi_i \in \g_{-i}^\theta \right\},
\end{align*}
is a finite-dimensional sub-algebra of $\wp_\theta$ and is isomorphic to a reductive subalgebra of $\g$ containing $\t$.

\subsection{Parahoric Group Schemes and Torsors} \label{PGST}

Let $X$ be a smooth (irreducible) curve.  We use the notation of Section \ref{PBRS}.  A group scheme $\GG$ will be called a \emph{parahoric Bruhat--Tits group scheme} if there exists a finite set $Z \subseteq X$ such that for each $x \in Z$, there is some $\theta_x \in \t = Y(T) \otimes_\Z \R$ such that
\begin{align*}
\GG|_{X \setminus Z} & \cong ( X \setminus Z) \times G, & \text{ and } \GG|_{\DD_x } & \cong \PPP_{\theta_x},
\end{align*}
for each $x \in Z$.  If $\boldsymbol{\theta} := \{ \theta_x \, : \, x \in Z \}$, then this group scheme will be written $\GG_{\boldsymbol{\theta}}$.  The stack of $\GG_{\boldsymbol{\theta}}$-torsors over $X$ will be denoted
\begin{align*}
\Bun_{\GG_{\boldsymbol{\theta}}} X.
\end{align*}
For simplicity, we will only be considering the case when $Z = \{ x \}$ is a singleton, in which case we write $\theta$ for $\theta_x$ and for $\boldsymbol{\theta}$.

Let $X$ and $x \in X$ be as above.  Let $s \in H^0(X, \OO_X(x))$ be a section vanishing (only) at $x$ and fix $r \in \N$.  Let $\XX := \XX_{\OO_X(x), s, r}$ be the associated root stack.

Fix a local type $\tau$ for $G$-bundles over $\XX$ (i.e.\ an equivalence class of representations $\mmu_r \to G$) and choose a representative $\rho$.  We may assume that $\rho$ is the restriction of a cocharacter $\theta : \Gm(L) \to T(L)$ to $\mmu_r(L) \subseteq \Gm(L)$ so we may associate to $\tau$ an element
\begin{align*}
\theta = \theta_\tau \in Y\big( T(L) \big) = Y\big( T( \K_x) \big) \otimes_\Z \tfrac{1}{r} \Z.
\end{align*}

\begin{prop}
Assume that one of the following three conditions holds for $\XX_{(L,r,s)}$:
\begin{itemize}
\item $g \geq 1$,

\item $g = 0$, and the support of the divisor for $s$
has cardinality at least three, and

\item $g = 0$, and the divisor for $s$ is of the form $d(x+y)$,
where $x$ and $y$ are distinct points.
\end{itemize}
Then there is an equivalence of stacks
\begin{align*}
\Bun_G^\tau \XX \arsim \Bun_{\GG_\theta} X.
\end{align*}
\end{prop}

\begin{proof}
We may choose a Galois cover $p : Y \to X$ as in Corollary \ref{globalquotient} so that we have $\Bun_G^\tau \XX \cong \Bun_{\Gamma, G}^\tau Y$ from Proposition \ref{GammaGlocaltype}.  But $\Bun_{\Gamma, G}^\tau Y \cong \Bun_{\GG_\theta} X$, by \cite[Theorem 5.3.1]{BalajiSeshadri2012}.
\end{proof}

\subsection{Parahoric Connections}

\subsubsection{Local Connections:  The Residue Condition}

Fix $\theta \in \t_\R$ and consider the parahoric subgroup $\PPP_\theta$ of $G(K)$.  Then \cite[\S2.3]{Boalch_Parahoric} considers the subspace
\begin{align*}
\mathcal{A}_\theta := \wp_\theta \, \frac{dz}{z}
\end{align*}
of \emph{logarithmic parahoric} or \emph{logahoric connections} of the space $\g(K) \, dz$, of meromorphic connections over the trivial $G$-bundle over the formal disc.  Observe that the space of such connections only depends on the Lie algebra $\wp_\theta$, rather than $\theta$ itself.

We will want to restrict this definition somewhat for our purposes by imposing a condition analogous to the second condition in (\ref{parcxn}).  Let $\widetilde{\omega} \, dz/z$ be a logahoric connection (for the choice of $\theta$), with $\widetilde{\omega} \in \wp_\theta$.  We may consider its weight zero piece,
\begin{align*}
\widetilde{\omega}_0 \in \g(K)_0^\theta.
\end{align*}

\begin{defn}
We will say that a logahoric connection \emph{satisfies the residue condition} if its weight zero piece is precisely $\theta$.  We will denote by $\mathcal{A}_\theta^\Res$ the space of logahoric connections satisfying the residue condition.
\end{defn}

This definition very much depends on the data of $\theta \in \t_\R$.

\begin{lem} \label{rescondn}
A logahoric connection $\widetilde{\omega} \, dz/z$ satisfies the residue condition if and only if
\begin{align*}
\widetilde{\omega} \in \theta + \t(z A) + \sum_{\alpha(\theta) \in \Z} \g_\alpha( z^{1 + m_\alpha(\theta)} A) + \sum_{\alpha(\theta) \not\in \Z} \g_\alpha( z^{m_\alpha(\theta)} A).
\end{align*}
\end{lem}

Let $B := A[t]/(t^r -z) = \C [\![ t ]\!]$ and let $L = \C(\!( t )\!)$ be its quotient field.  We consider a trivial $G$-bundle $E = \widehat{\DD} \times G$ over $\widehat{\DD} := \Spec \, B$ with a compatible $\mmu_r$-action, the $\mmu_r$-action on $\widehat{\DD}$ coming from that on $B$, where it is given by $\gamma \cdot t = \zeta^{-1} t$, with $\gamma \in \mmu_r$ a fixed generator.  As above, this may be realized via a homomorphism $\theta : \mmu_r \to G(L)$, which we may assume factors through $T(L) \subseteq G(L)$.  In fact, we may think of $\theta$ as the restriction of a cocharacter in $Y(T(L))$, which we lazily also denote by $\theta$.  We have
\begin{align*}
Y(T(L)) = Y(T(K)) \otimes_\Z \tfrac{1}{r} \Z \subseteq Y(T(K))
\end{align*}
and so we will think of $\theta$ as an element of $Y(T(K)) \otimes_\Z \Q$.

\begin{prop} \label{logastackcxn}
Logahoric connections satisfying the residue condition for $\theta$ (i.e., elements of $\mathcal{A}_\theta^\Res$) are in a one-to-one correspondence with $\mmu_r$-connections on the trivial $G$-bundle over $\widehat{\DD}$, for which the $\mmu_r$-action is given by the homomorphism $\theta$.
\end{prop}

\begin{proof}
A connection $\omega$ on $E$ is given by a $\g$-valued $1$-form on $W$; it may be written 
\begin{align*}
\omega = \widetilde{\omega} \, dt = \left( \sum_{i=1}^l \omega_i h_i + \sum_{\alpha \in \Phi} \omega_\alpha x_\alpha \right) dt,
\end{align*}
for some $\omega_i, \omega_\alpha \in B$, where $h_1, \cdots, h_l$ is a basis of 
$\t$.

The condition for $\omega$ to be $\mmu_r$-invariant (cf.\ proof of Proposition \ref{rootstackconnection}) is that $\gamma^* \omega = \omega$, or
\begin{align} \label{omegainvariance}
\Ad \, \theta(\gamma) (\gamma \cdot \widetilde{\omega}) \, (\gamma \cdot dt) = \widetilde{\omega} \, dt.
\end{align}
For $1 \leq i \leq l$, this translates to 
\begin{align*}
\omega_i(t) \, dt = \zeta^{-1} \omega_i( \zeta^{-1} t) \, dt,
\end{align*}
and it follows that 
\begin{align*}
\omega_i (t) \, dt = \nu_i(z) \, dz
\end{align*}
for some $\nu_i \in A$.  For $\alpha \in \Phi$, (\ref{omegainvariance}) implies
\begin{align*}
\omega_\alpha(t) \, dt = \alpha\big( r\theta(\gamma) \big) \zeta^{-1} \omega_\alpha( \zeta^{-1} t) \, dt = \zeta^{ r\alpha(\theta) -1} \omega_\alpha( \zeta^{-1} t) \, dt.
\end{align*}
In this case, one sees that
\begin{align*}
\omega_\alpha \, dt = \begin{cases} \nu_\alpha \, dz, & \text{ if } m_\alpha(\theta) = -\alpha(\theta), \text{ i.e.}, \alpha(\theta) \in \Z \\ t^{r \alpha(\theta)} z^{ m_\alpha(\theta) - 1} \nu_\alpha \, dz, & \text{ if } \alpha(\theta) + m_\alpha(\theta) > 0 , \end{cases}
\end{align*}
for some $\nu_\alpha \in A$, where one recalls that $m_\alpha(\theta) = -\lfloor \alpha(\theta) \rfloor$.  Therefore,
\begin{align*}
\omega = \left( \sum_{i=1}^l z \nu_i h_i + \sum_{\alpha(\theta) \in \Z} z \nu_\alpha x_\alpha + \sum_{\alpha(\theta) \not\in \Z} t^{r \alpha(\theta)} z^{m_\alpha(\theta)} \nu_\alpha x_\alpha \right) \frac{dz}{z}.
\end{align*}

Now, using the change of frame $t^\theta$, the connection form becomes 
\begin{align*}
\Ad \, t^{-\theta}( \omega) + t^{-\theta} d(t^{\theta}) = \left( \theta + \sum_{i=1}^l z \nu_i h_i + \sum_{\alpha(\theta) \in \Z} z^{1 + m_\alpha(\theta)} \nu_\alpha x_\alpha + \sum_{\alpha(\theta) \not\in \Z} z^{m_\alpha(\theta)} \nu_\alpha x_\alpha \right) \frac{dz}{z},
\end{align*}
and from Lemma \ref{rescondn}, we find that the induced connection lies in $\mathcal{A}_\theta^\Res$.  Conversely, a logahoric connection satisfying the residue condition is of the above form, and reversing the process (via the change of frame $t^\theta$), we recover a connection on the trivial $G$-bundle over $W$ with no poles.
\end{proof}

\subsubsection{Global Connections}

Let $E \to X$ be a $\GG_\theta$-torsor.  This may be given by a $G$-bundle $E|_{X \setminus x}$ over $X \setminus x$, a parahoric torsor $E|_{\DD_x}$ over $\DD_x$ and an isomorphism $\eta : E|_{X \setminus x}|_{\DD_x^\times} \arsim E|_{\DD_x} |_{\DD_x^\times}$ (the reader may wish to have another look at the diagram (\ref{Xx})).  We define a \emph{connection on $E$} to be a pair consisting of a connection $\omega_0$ on the $G$-bundle $E|_{X \setminus x}$ and a logahoric connection $\omega_x$ on $E|_{\DD_x}$ satisfying the residue condition such that 
\begin{align*}
\omega_0|_{\DD_x^\times} = \eta^* \omega_x|_{\DD_x^\times}.
\end{align*}

\begin{prop}
Let $\E$ be a $G$-bundle on $\XX$ and let $E$ be the corresponding parahoric bundle on $X$.  Then there is a one-to-one correspondence between connections on $\E$ and connections on $E$.
\end{prop}

\begin{proof}
It is clear that if we begin with a connection $\omega$ on a $G$-bundle $\E$ over $\XX$, then we can take restrictions to $(X \setminus x) \times_X \XX \cong X \setminus x$ and to $\DD_x \times_X \XX \cong{ \widehat{\DD} }$, the latter being a connection compatible with the $\mmu_r$-action, and hence by Proposition \ref{logastackcxn}, yields a logahoric connection on the parahoric torsor over $\DD_x$.  The isomorphism over $\DD_x^\times$ comes from the fact that $\omega$ is defined over all of $\XX$.

Conversely, given $\omega_0, \omega_x$ as in the definition, the pullback of $\omega_x$ to $\E|_{\DD_x \times_X \XX }$ gives a well-defined connection (again by Proposition \ref{logastackcxn}).  The isomorphism $\eta$ gives patching data over $\E|_{ \DD_x^\times \times_X \XX }$, and so we get a connection over $\E$.
\end{proof}

One cheaply obtains the following.

\begin{cor}
Let $E \to X$ be a $\GG_\theta$-torsor.  Then it admits a connection in the above sense if and only if the corresponding $G$-bundle on $\XX$ satisfies the condition of Theorem \ref{connectioncondition}.
\end{cor}

\small

\bibliographystyle{amsalpha}

\end{document}